\documentclass{article}

\usepackage{latexsym,amsmath,amsfonts,amscd, amsthm, dsfont}
\usepackage{bm,color}
\usepackage{epsfig,verbatim,epstopdf,graphics,graphicx}
\usepackage{changebar,multirow}
\usepackage{subcaption}
\usepackage{cite}

\usepackage{algorithmic}
\usepackage[bottom]{footmisc}
\usepackage{yhmath,booktabs,tikz,verbatim}
\usepackage{authblk,enumitem}
\usepackage{comment}
\usetikzlibrary{arrows,backgrounds,shapes}
\usetikzlibrary{decorations.pathmorphing}
\usepackage{hyperref}
\usepackage{algorithm2e}
\usepackage{algorithmic}
\usepackage{orcidlink}

\numberwithin{equation}{section}

\graphicspath{{./}{./figure/}}
\allowdisplaybreaks

\newtheoremstyle{plainNoItalics}{}{}{\normalfont}{}{\bfseries}{.}{ }{}

\theoremstyle{plain}
\newtheorem{thm}{Theorem}[section]

\theoremstyle{plainNoItalics}

\newtheorem{lem}[thm]{Lemma}

\newtheorem{exa}[thm]{Example}

\renewcommand{\theequation}{\thesection.\arabic{equation}}
\renewcommand{\thefigure}{\thesection.\arabic{figure}}
\renewcommand{\thetable}{\thesection.\arabic{table}}

\def\SS{\mathbf{S}}

\newcommand{\mD}{{\mathcal D}}
\newcommand{\mtD}{\widetilde{\mathcal D}}
\newcommand{\mP}{{\mathcal P}}
\newcommand{\mO}{{\mathcal O}}

\newcommand{\mttD}{\widetilde{\widetilde{\mathcal D}}}

\newcommand{\be}{\begin{eqnarray}}
\newcommand{\ee}{\end{eqnarray}}
\newcommand{\beno}{\begin{eqnarray*}}
\newcommand{\eeno}{\end{eqnarray*}}

\makeatletter

\newcommand{\Rmnum}[1]{\expandafter\@slowromancap\romannumeral #1@}
\makeatother
\newcommand{\rev}[1]{\textcolor{black}{#1}}

\begin{document}

\title{Boundary corrections for kernel approximation to differential operators \thanks{The research of the authors was supported by AFOSR grants FA9550-17-1-0394 and NSF grant DMS-1912183. The first two authors were also supported by AFOSR grants FA9550-24-1-0254, ONR grant N00014-24-1-2242,  and DOE grant DE-SC0023164. The third author was also supported by the National Research Foundation of Korea(NRF) grant funded by the Korea government(MSIT) NRF-2022R1F1A1066389.}}
\date{}
\author{Andrew Christlieb${}^{1,*}$\href{https://orcid.org/0000-0002-5395-5455}{\protect\orcidlogo}, Sining Gong${}^{1,\dagger}$\href{https://orcid.org/0000-0002-8049-2302}{\protect\orcidlogo} and Hyoseon Yang${}^{2}$\href{https://orcid.org/0000-0002-9847-3716}{\protect\orcidlogo}}
\footnotetext[1]{Department of Computational Mathematics, Science and Engineering, Michigan State University, East Lansing, MI, 48824, United States. ${}^*$christli@msu.edu, ${}^\dagger$gongsini@msu.edu; corresponding author.}
\footnotetext[2]{Department of Mathematics, College of Sciences, Kyung Hee University, Seoul 02447, South Korea, hyoseon@khu.ac.kr.}

\maketitle

\begin{abstract}  
 \rev{The kernel-based} approach to operator approximation for partial differential equations has been shown to be  unconditionally stable for linear PDEs and numerically exhibit unconditional stability for non-linear PDEs.  These methods have the same computational cost as an explicit  finite difference scheme but can exhibit order reduction at boundaries.  In previous work on  periodic domains, \cite{christlieb2019kernel, christlieb2020}, order reduction was addressed, yielding high-order accuracy.
 The issue addressed in this work is the elimination of order reduction of the kernel-based approach for a more general set of boundary conditions. 
 Further,  we consider the case of both first and second order operators.  To demonstrate the theory, we provide not only the mathematical proofs but also experimental results by applying various boundary conditions to different types of equations.  The results agree with the theory, demonstrating a systematic path to high order for kernel-based methods on bounded domains. 
\end{abstract}


\pagestyle{myheadings} \thispagestyle{plain}\markboth{}{}
\section{Introduction}

Kernel-based approximation to operators  for partial differential equations has been shown to
be unconditionally stable for linear PDEs and numerically exhibit unconditional stability for non-linear
PDEs \cite{jia2008krylov,causley2014higher,causley2016method,causley2017method,christlieb2020,christlieb2019kernel, christlieb2020kernel}. The potential advantage of such an approach is that they  offer the computational cost of an  explicit finite difference method while providing a path to doing complex geometry with a large time step \cite{causley2014higher,thavappiragasam2020fast,christlieb2020kernel,christlieb2021parallel}.  

In this paper, we propose a kernel type approximation method to differential operators to solve several types of PDEs with various boundary conditions. In order to develop the kernel-based approximation, it is necessary to understand the Method of Lines Transpose (MOL$^T$) scheme, which is also known as Rothes's method \cite{schemann1998adaptive, salazar2000theoretical, causley2014method}.
Traditional numerical schemes for time dependent PDEs fall under the Method of Lines (MOL) framework, starting with discretization of the spatial operators and solving the resulting initial value problems (IVPs) by integrating with an appropriate time integrator.
Alternatively, the MOL$^T$ approach discretizes the temporal variable first. Then, the resulting boundary value problems (BVPs) are solved at discrete time levels. To address the BVPs, the continuous operator (in space) is analytically inverted using an integral solution. 

Recent advancements in the MOL$^T$ have focused on extending the method to address more general nonlinear partial differential equations (PDEs), such as the nonlinear degenerate convection-diffusion equations \cite{christlieb2020} and the Hamilton-Jacobi equations \cite{christlieb2019kernel, christlieb2020kernel}.  
The main contribution in these papers was to leverage the linearity of a specified \textit{differential operator}, rather than necessitating linearity in the underlying equations. This approach enabled the expression of derivative operators in the problems using kernel representations designed for linear problems.  This perspective  opened up the possibility of \rev{pairing} the unconditionally stable derivative operators for the spatial derivative with any explicit MOL time stepping schema,
such as the strong-stability-preserving Runge-Kutta (SSP-RK) methods \cite{gottlieb2001strong}. 
By formulating the derivative operators in this manner, the stability of the explicit time marching schemes was enhanced, as a global coupling was introduced via the integral operator.

Starting from the kernel-based high-order approach for simulating the H-J equations with non-periodic boundary treatment \cite{christlieb2019kernel}, this paper constructs generalized differential operators by incorporating appropriate boundary derivative terms, ensuring that the scheme achieves high order for a range of boundary conditions, including periodic, Dirichlet, and Neumann conditions. As part of the generalization introduced in this work, the method is extended to the case of high order kernel-based methods for second derivative operators.
The construction of these operators for non-periodic domains is more complex compared to periodic domains, especially when dealing with second derivatives \cite{christlieb2020}.  To achieve high order kernel-based operators for a range of boundary conditions, we extend the ideas in \cite{christlieb2019kernel} and make use of Taylor expansions to derive recurrent relations for general boundary terms that need to be constructed for the method to achieve the desired high order accuracy. 


To the best of our knowledge, this seems to be the first paper to theoretically achieve high-order accuracy for both first and second derivatives under non-periodic boundary conditions for kernel-based methods. Our operators can be applied to a variety of partial differential equations (PDEs), including transport, wave, heat, and convection-diffusion equations.

The paper is organized as follows. After we review the kernel-based representations for first and second order derivative operators, 
we analyze them to obtain the generalized boundary terms for both first and second order derivative operators in Section 2. In Section 3, we present the boundary corrected operators for general boundary conditions and address periodic and non-periodic conditions, such as Dirichlet and Neumann problems. Before we proceed to numerical examples, we demonstrated the implementation of spatial discretization on both uniform and non-uniform grids in Section 4. Then, a collection of numerical examples is presented to demonstrate the performance of the proposed method in Section 5.

\section{Defining Generalized Boundary Terms} \label{sec:newbdyterm}
In this section, we establish that the kernel-based approximation provides a low order approximation to first and second order differential operators on a bounded domain, motivating the rest of the paper.
To begin, we review the structure of the kernel-based approximations, which are derived from the MOL$^T$ framework.

\subsection{Kernel-based differential operators}
We first recall
the kernel-based representation of the first spatial derivative $\partial_x$ and second spatial derivative $\partial_{xx}$, 
following \cite{christlieb2020, causley2014higher}. These representations are derived from the successive convolution of the underlying kernel functions and will play a key role in the proposed scheme. We introduce two function operators $\mathcal{L}_{L}$ and $\mathcal{L}_{R}$ to account for waves traveling in 
opposite 
directions, downwinding and upwinding, respectively:
\begin{align}\label{def:Lops}
\mathcal{L}_{L} = \mathcal{I} - \frac{1}{\alpha}\partial_{x}, \quad
\mathcal{L}_{R} = \mathcal{I} + \frac{1}{\alpha}\partial_{x},  \quad x \in [a,b],
\end{align} 
where $\mathcal{I}$ is an identity operator and  $\alpha>0$ is a constant. Using an integrating factor, the inversion of these operators can be written by
\begin{align}\label{def:Linvs}
\begin{split}
	& \mathcal{L}_{L}^{-1}[v](x) = I_{L}[v](x) + B_L e^{-\alpha (b - x)}, \\
	& \mathcal{L}_{R}^{-1}[v](x) = I_{R}[v](x) + A_R e^{-\alpha (x - a)},
\end{split}
\end{align}
for a function $v(x)$ defined on $[a,b]$. Here $I_L$ and $I_R$ are the integral operators given by
\begin{align}\label{def:Int_ops}
\begin{split}
	& I_{L}[v](x) = \alpha \int_{x}^{b} e^{-\alpha (s-x)}v(s) \, ds, \\
	& I_{R}[v](x) = \alpha \int_{a}^{x} e^{-\alpha (x-s)}v(s) \, ds,
\end{split}
\end{align}
and $A_{R}$ and $B_{L}$ are the constants determined by the boundary conditions.
Then we introduce the operators
\begin{equation}\label{def:D_ops}
\mathcal{D}_{L} = \mathcal{I} - \mathcal{L}^{-1}_{L}, \quad
\mathcal{D}_{R} = \mathcal{I} - \mathcal{L}^{-1}_{R}, \quad x\in[a,b].
\end{equation}
Each of these can be expanded in a Neumann series:
\begin{align}\label{eq:sum}
\begin{split}
	& \frac{1}{\alpha}\partial_{x}^{+} = \mathcal{I}-\mathcal{L}_{L}
	= \mathcal{L}_{L} (\mathcal{L}^{-1}_{L}-\mathcal{I})
	= -\mathcal{D}_{L}/(\mathcal{I} -\mathcal{D}_{L} )
	= -\sum_{p=1}^{\infty}\mathcal{D}_{L}^{p}, \\
	& \frac{1}{\alpha}\partial_{x}^{-} = \mathcal{L}_{R}-\mathcal{I}
	= \mathcal{L}_{R} (\mathcal{I}-\mathcal{L}_{R}^{-1})
	= \mathcal{D}_{R}/(\mathcal{I} -\mathcal{D}_{R} )
	= \sum_{p=1}^{\infty}\mathcal{D}_{R}^{p}.
\end{split}
	\end{align}
Here $\partial_{x}^{+}$ and $\partial_x^-$ indicate the left-sided and right-sided approximations of the derivative in $x$, respectively, along an interface.

We similarly introduce a function operator $\mathcal{L}_{0}$ for the second derivative case:
\begin{align*}
\mathcal{L}_{0} = \mathcal{I} - \frac{1}{\alpha^2}\partial_{xx}, \quad x \in [a,b],
\end{align*} 
which can be viewed as the analogue of $\mathcal{L}_{L}$ and $\mathcal{L}_{R}$ for second order derivatives.
The inverse of $\mathcal{L}_{0}$ can be given by an integral formula with two boundary-associated terms:
\begin{align}\label{def:L0_invs}
\begin{split}
	\mathcal{L}_{0}^{-1}[v](x) = I_{0}[v](x) + B_0 e^{-\alpha (b - x)} + A_0 e^{-\alpha (x - a)},
\end{split}
\end{align}
for a function $v(x)$ defined on $[a,b]$. Here $I_0$ is the integral operator defined as
\begin{align}\label{def:Int0_ops}
\begin{split}
	I_{0}[v](x) = \frac{\alpha}{2} \int_{a}^{b} e^{-\alpha |s-x|}v(s) \, ds, 
\end{split}
\end{align}
and $A_{0}$ and $B_{0}$ are the the constants  determined by boundary conditions imposed.
We then define the operator
\begin{equation}\label{def:D0_ops}
\mathcal{D}_{0} = \mathcal{I} - \mathcal{L}^{-1}_{0}, \quad x\in[a,b],
\end{equation}
and expand it into a Neumann series as well:
	\begin{align}
	\label{eq:D0sum}
	& \frac{1}{\alpha^2}\partial_{xx} = \mathcal{I}-\mathcal{L}_{0}
	= \mathcal{L}_{0} (\mathcal{L}^{-1}_{0}-\mathcal{I})
	= -\mathcal{D}_{0}/(\mathcal{I} -\mathcal{D}_{0} )
	= -\sum_{p=1}^{\infty}\mathcal{D}_{0}^{p}.
	\end{align}
These expansions have been coined successive convolution (or resolvent expansions). In \cite{causley2013method,causley2016method}, we leverage properties of the resolvent expansions to establish stability and convergence for advection, wave and parabolic equations in periodic domains. Here we focus on a generalized approach to consistency for bounded domains and leave stability to later work.

\subsection{Differential operator \texorpdfstring{$\partial_x$}{TEXT} with general boundary terms}
In this subsection, we analyze the accuracy of the kernel-based representations for the first derivative on a bounded domain. 
From the equation \eqref{eq:sum}, the 
differential operator $\partial_x$ can be approximated by
\begin{align}
     \partial_{x}^{+} =-{\alpha} \sum_{p=1}^{\infty}\mathcal{D}_{L}^{p}, \quad \partial_{x}^{-} = \alpha \sum_{p=1}^{\infty}\mathcal{D}_{R}^{p}
\end{align}
using \eqref{def:Linvs} and \eqref{def:D_ops},
\begin{align}\label{eq:D_ops}
\begin{split}
    & \mathcal{D}_{L} = \mathcal{I} - (I_{L}[v](x) + B_L e^{-\alpha (b - x)}), \\
    & \mathcal{D}_{R} = \mathcal{I} - (I_{R}[v](x) + A_R e^{-\alpha (x - a)}), 
\end{split}
\end{align}
where $I_{L}$ and $I_{R}$ are the integral operators in \eqref{def:Int_ops}.
Now, we derive new boundary terms $B_L$ and $A_R$ for generalizing the derivative operators \rev{by revisiting the derivation of kernel-based operators.}

We firstly consider the linear wave equation
\begin{equation*}
	\partial_t v - c\partial_x v = 0,
\end{equation*}
and using the backward Euler time discretization, we obtain
\begin{equation}\label{eq:lin_semi}
	(1 - c \Delta t \partial_x ) v^{n+1} = v^n
\end{equation}
where $v^{n} := v(t^n)$ and  $v^{n+1} := v(t^{n+1})$. 
Using the operator $\mathcal{L}_L$ from \eqref{def:Lops} with $\alpha := \frac{1}{c\Delta t}$ for $c\geq 0$, 
the equation \eqref{eq:lin_semi} can be compactly stated as 
\begin{equation}\label{eq:L}
	\mathcal{L}_L [v^{n+1}] = v^n.
\end{equation}
Then the numerical solution is updated from
\begin{equation}\label{eq:Linv}
	 v^{n+1}(x) = \alpha \int_x^b e^{-\alpha (y-x)} v^n(y) dy + v^{n+1}(b) e^{-\alpha (b-x)}.
\end{equation}
From the equations \eqref{eq:L} and \eqref{eq:Linv}, we can find inverse of $\mathcal{L}$ operator with general boundary terms:
\begin{align}\label{def:LLinv}
\begin{split}
    \mathcal{L}^{-1}_L[(\cdot)^n] (x)  &=  (\cdot)^{n+1} (x) \\
    &=  \alpha \int_x^b e^{-\alpha (y-x)}(\cdot)^n(y) dy +(\cdot)^{n+1}(b) e^{-\alpha (b-x)}.
\end{split}
\end{align}
Putting $B_L:=B_L[(\cdot)^n]=(\cdot)^{n+1}(b)$ as in \eqref{def:Linvs}, this boundary term can be approximated by a Taylor expansion:
\begin{align}\label{term:B}
\begin{split}
	B_L[v^n] = v^{n+1}(b) &= v^n(b) + \Delta t v^n_t(b) + \mathcal{O}(\Delta t^2) \\
	&= v^n(b) + c \Delta t v^n_x(b) + \mathcal{O}(\Delta t^2) \\
	&= v^n(b) + \frac{1}{\alpha} v^n_x(b) + \mathcal{O}(\Delta t^2).
\end{split}
\end{align}
The operator $\mathcal{L}^{-1}_R$ with a boundary term $A_R$ is derived by a similar process,
\begin{align}\label{def:LRinv}
\begin{split}
    \mathcal{L}^{-1}_R[(\cdot)^n] (x) &=  \alpha \int_a^x e^{-\alpha (x-a)}(\cdot)^n(y) dy + A_R e^{-\alpha (x-a)}
\end{split}
\end{align}
with
\begin{align}\label{term:A}
\begin{split}
	A_R[v^n] = v^{n+1}(a) &= v^n(a) + \Delta t v^n_t(a) + \mathcal{O}(\Delta t^2) \\
	&= v^n(a) - \frac{1}{\alpha} v^n_x(a) + \mathcal{O}(\Delta t^2).
\end{split}
\end{align}

Now, we can derive the analysis for the operators $\mathcal{D}_{L}$ and $\mathcal{D}_{R}$ with the general boundary terms. The result is summarized in the following lemma.

\begin{lem}\label{lem_D}
     Suppose $v\in \mathcal{C}^{k+1}[a,b]$ and we set the operator $\mD_{L}$ and $\mD_{R}$ in \eqref{def:D_ops} with general boundary treatment \eqref{term:B} and \eqref{term:A}. 
	Then, we can obtain that
	\begin{align}\label{eq:D_order}
	\begin{split}
		\mathcal{D}_{L}[v](x)
		=& -\sum_{p=1}^{k} \left(\frac{1}{\alpha}\right)^{p}\partial_{x}^{p}v(x) + \sum_{p=2}^{k}\left(\frac{1}{\alpha}\right)^{p}\partial_{x}^{p}v(b) e^{-\alpha(b-x)} -\left(\frac{1}{\alpha}\right)^{k+1}I_{L}[\partial_{x}^{k+1}v](x), \\
		\mathcal{D}_{R}[v](x)
		=& -\sum_{p=1}^{k} \left(-\frac{1}{\alpha}\right)^{p}\partial_{x}^{p}v(x) + \sum_{p=2}^{k}\left(-\frac{1}{\alpha}\right)^{p}\partial_{x}^{p}v(a) e^{-\alpha(x-a)} -\left(-\frac{1}{\alpha}\right)^{k+1}I_{R}[\partial_{x}^{k+1}v](x)
	\end{split}
	\end{align}
\rev{    where $I_{L}$ and $I_{R}$ are the integral operators in \eqref{def:Int_ops}.}
\end{lem}

\begin{proof}
	We consider the case of $\mathcal{D}_{L}$.	Using the definition of $I_L$ and integration by parts, we have
\begin{align*}
\textcolor{blue}{I_L[v](x) = \frac{1}{\alpha} I_L[v'](x) + v(x) - v(b) e^{-\alpha(b - x)}.}
\end{align*}
\rev{This identity allows us to express \( \mD_L[v](x) = v(x) - \mathcal{L}_L^{-1}[v](x) \) in terms of derivatives and boundary contributions.
By repeating this process recursively, we obtain}
\begin{align*}
\mD_L[v](x) = -\sum_{p=1}^k \left( \frac{1}{\alpha} \right)^p \partial_x^p v(x) + \sum_{p=2}^k \left( \frac{1}{\alpha} \right)^p \partial_x^p v(b) e^{-\alpha(b-x)} - \left( \frac{1}{\alpha} \right)^{k+1} I_L[\partial_x^{k+1}v](x).    
\end{align*}

For $\mathcal{D}_{R}$, the proof can be established by a similar process and the lemma is proved.  

\end{proof}


\subsection{Differential operator \texorpdfstring{$\partial_{xx}$}{TEXT} with general boundary terms}
Similar to the previous section, from the linear heat equation
\begin{equation}\label{eq:l0in}
	\partial_{t} v - c^2\partial_{xx} v = 0,
\end{equation}
using the backward Euler time discretization, we obtain
\begin{equation}\label{eq:l0in_semi}
	(1 - (c^2\Delta t) \partial_{xx} ) v^{n+1} = v^n
\end{equation}
where $v^{n} := v(t^n)$ and  $v^{n+1} := v(t^{n+1})$. 
We define the operator $\mathcal{L}_0$ as 
\begin{equation}
	\mathcal{L}_0 = \mathcal{I} - \frac{1}{\alpha^2}\partial_{xx},
\end{equation}
with $\alpha := \frac{1}{c\sqrt{\Delta t}}$ so that 
the equation \eqref{eq:l0in_semi} can be compactly stated as 
\begin{equation}\label{eq:0}
	\mathcal{L}_0 [v^{n+1}] = v^n.
\end{equation}
By doing integration by parts twice and we have the numerical solution updated from
\[
v^{n+1}(x) = \frac{\alpha}{2} \int_a^b e^{-\alpha |y-x|} v^n(y) dy + \frac{1}{2}(v^{n+1}(b)+\frac{1}{\alpha}v^{n+1}_x(b)) e^{-\alpha (b-x)} + \frac{1}{2}(v^{n+1}(a)-\frac{1}{\alpha}v^{n+1}_x(a)) e^{-\alpha (x-a)},
\]
then we can have the following:
\begin{equation}\label{eq:0inv}
	 v^{n+1}(x) = \frac{\alpha}{2} \int_a^b e^{-\alpha |y-x|} v^n(y) dy + B_0 e^{-\alpha (b-x)} + A_0 e^{-\alpha (x-a)} = I_0[v^{n}](x)+ B_0 e^{-\alpha (b-x)} + A_0 e^{-\alpha (x-a)},
\end{equation}
where the general boundary operator is:
\begin{align}\label{gen_bdy_AB0}
\begin{split}
     B_0 & = \frac{1}{2}(v^{n+1}(b)+\frac{1}{\alpha}v^{n+1}_x(b)) \\
    A_0 & = \frac{1}{2}(v^{n+1}(a)-\frac{1}{\alpha}v^{n+1}_x(a)).
\end{split}
\end{align}
Now depending on the given boundary information we could furthermore provide $A_0$ and $B_0$ in different format. For example, for the Dirichlet boundary condition, we have:
\begin{align}\label{Dir_bdy_AB0}
\begin{split}
   B_0 & =  \frac{1}{1-\mu^2} (v^{n+1}(b) - \mu v^{n+1}(a) + \mu  I_0[v^{n}](a) -  I_0[v^{n}](b) )\\
   A_0 &=  \frac{1}{1-\mu^2} (v^{n+1}(a) - \mu v^{n+1}(b) + \mu I_0[v^{n}](b) -   I_0[v^{n}](a) ),
\end{split}
\end{align}
with $\mu = e^{-\alpha (b-a)}$; while for the Neumann boundary condition, we have:
\begin{align}\label{Neu_bdy_AB0}
\begin{split}
   B_0 & =  \frac{1}{\alpha(\mu^2-1)} (\mu v_x^{n+1}(a) - v_x^{n+1}(b) - \mu \alpha I_0[v^{n}](a) -  \alpha I_0[v^{n}](b) )\\
   A_0 &=  \frac{1}{\alpha(\mu^2-1)} (v_x^{n+1}(a) - \mu v_x^{n+1}(b) - \mu \alpha I_0[v^{n}](b) -  \alpha I_0[v^{n}](a) ).
\end{split}
\end{align}
From the equations \eqref{eq:0} and \eqref{eq:0inv}, we can find definition of $\mathcal{L}$ operator:
\begin{align}\label{def:L0inv}
\begin{split}
    \mathcal{L}^{-1}_0[(\cdot)^n] (x)  &=  (\cdot)^{n+1} (x) \\
    &=  \frac{\alpha}{2}\int_a^b e^{-\alpha |y-x|}(\cdot)^n (y) dy + A_0 e^{-\alpha (x-a)} + B_0 e^{-\alpha (b-x)}
\end{split}
\end{align}
Now the boundary term can be approximated by a Taylor expansion and the equation \eqref{eq:l0in}:
\begin{align}\label{term:B0}
\begin{split}
	v^{n+1}(b) &= v^n(b) + (\Delta t) v^n_{t}(b) + \mathcal{O}(\Delta t^2) \\
	&= v^n(b) + (c^2 \Delta t) v^n_{xx}(b) + \mathcal{O}(\Delta t^2) \\
	&= v^n(b) + \frac{1}{\alpha^2} v^n_{xx}(b) + \mathcal{O}(\Delta t^2),
\end{split}
\end{align}
and we also have similar result for $v^{n+1}(a)$. 
Then we can define $\mathcal{D}$ operators in \eqref{def:D0_ops}
\begin{align}\label{eq:D0_ops}
\begin{split}
    \mathcal{D}_{0}[(\cdot)](x) = (\mathcal{I} - \mathcal{L}^{-1}_{0})[(\cdot)](x) =
    (\cdot)(x) - \frac{\alpha}{2} \int_a^b e^{-\alpha |y-x|}(\cdot)^n(y) dy - B_0 e^{-\alpha (b-x)}- A_0 e^{-\alpha (x-a)}.
\end{split}
\end{align}

\begin{lem}\label{lem_D2}
  Suppose $v\in \mathcal{C}^{2k+2}[a,b]$ and we set the operator $\mD_{0}$ in \eqref{def:D0_ops} with general boundary treatment \eqref{gen_bdy_AB0} with \eqref{term:B0}. 
	Then, we can obtain that
	\begin{align}\label{eq:D0_order} 
	\begin{split}
		\mathcal{D}_{0}[v](x)
        & =  -\sum\limits_{p=1}^{k} (\frac{1}{\alpha})^{2p} \partial_{x}^{2p}v(x) - (\frac{1}{\alpha})^{2k+2} I_{0}[\partial_{x}^{2k+2}v](x) \\
      & \quad + \frac{1}{2}\sum\limits_{p=4}^{2k+1} (-\frac{1}{\alpha})^{p} \partial_{x}^{p}v(a) e^{-\alpha(x-a)}
     + \frac{1}{2}\sum\limits_{p=4}^{2k+1} (\frac{1}{\alpha})^{p} \partial_{x}^{p}v(b) e^{-\alpha(b-x)}
	\end{split}
	\end{align}
\rev{where $I_{0}$ is the integral operator in \ref{def:Int0_ops}.}
\end{lem}

\begin{proof}
 We know that 
 \begin{align*}
     I_0[v](x) & = \frac{1}{\alpha^2}I_0[v_{xx}](x) + v(x) - \frac{1}{2}(v(a)-\frac{1}{\alpha}v_x(a))e^{-\alpha(x-a)} - \frac{1}{2}(v(b)+\frac{1}{\alpha}v_x(b))e^{-\alpha(b-x)}
     \end{align*}
\rev{and by repeating the integration by parts recursively,}
\begin{align*}
I_0[v](x) 
     & = \sum\limits_{p=0}^{k} (\frac{1}{\alpha})^{2p} \partial_{x}^{2p}v(x)  + (\frac{1}{\alpha})^{2k+2} I_{0}[\partial_{x}^{2k+2}v](x)  \\
     & \quad 
     - \frac{1}{2}\sum\limits_{p=0}^{2k+1} (-\frac{1}{\alpha})^{p} \partial_{x}^{p}v(a) e^{-\alpha(x-a)}
     - \frac{1}{2}\sum\limits_{p=0}^{2k+1} (\frac{1}{\alpha})^{p} \partial_{x}^{p}v(b) e^{-\alpha(b-x)}.
 \end{align*}
 Then, 
 \begin{align*}
     \mD_0[v](x) & = v(x) - I_0[v](x) - A_0e^{-\alpha(x-a)} - B_0e^{-\alpha(b-x)} \\
     & = -\sum\limits_{p=1}^{k} (\frac{1}{\alpha})^{2p} \partial_{x}^{2p}v(x)  - (\frac{1}{\alpha})^{2k+2} I_{0}[\partial_{x}^{2k+2}v](x)  \\
     & \quad + \left(-A_0+ \frac{1}{2}\sum\limits_{p=0}^{2k+1} (-\frac{1}{\alpha})^{p} \partial_{x}^{p}v(a)\right) e^{-\alpha(x-a)}
     + \left(-B_0+ \frac{1}{2}\sum\limits_{p=0}^{2k+1} (\frac{1}{\alpha})^{p} \partial_{x}^{p}v(b) \right)e^{-\alpha(b-x)}.
 \end{align*}
 Note that with $v = v^n(x)$,
 \begin{align*}
     -B_0 & + \frac{1}{2}\sum\limits_{p=0}^{2k+1} (\frac{1}{\alpha})^{p} \partial_{x}^{p}v(b) = -\frac{1}{2}\sum\limits_{p=0}^{3} (\frac{1}{\alpha})^{p} \partial_{x}^{p}v(b) + \frac{1}{2}\sum\limits_{p=0}^{2k+1} (\frac{1}{\alpha})^{p} \partial_{x}^{p}v(b) =  \frac{1}{2}\sum\limits_{p=4}^{2k+1} (\frac{1}{\alpha})^{p} \partial_{x}^{p}v(b), \\
     -A_0 & +  \frac{1}{2}\sum\limits_{p=0}^{2k+1} (-\frac{1}{\alpha})^{p} \partial_{x}^{p}v(a) = \frac{1}{2}\sum\limits_{p=4}^{2k+1} (-\frac{1}{\alpha})^{p} \partial_{x}^{p}v(a).
 \end{align*}
Therefore, we get the representation of the operator:
 \begin{align*}
     \mD_0[v](x) & = v(x) - I_0[v](x) - A_0e^{-\alpha(x-a)} - B_0e^{-\alpha(b-x)} \\
     & = -\sum\limits_{p=1}^{k} (\frac{1}{\alpha})^{2p} \partial_{x}^{2p}v(x) - (\frac{1}{\alpha})^{2k+2} I_{0}[\partial_{x}^{2k+2}v](x) \\
      & \quad + \frac{1}{2}\sum\limits_{p=4}^{2k+1} (-\frac{1}{\alpha})^{p} \partial_{x}^{p}v(a) e^{-\alpha(x-a)}
     + \frac{1}{2}\sum\limits_{p=4}^{2k+1} (\frac{1}{\alpha})^{p} \partial_{x}^{p}v(b) e^{-\alpha(b-x)}.
 \end{align*}
\end{proof}

\section{Generalized Kernel Based Schemes}
In this section, we establish recurrent relations  for correcting the order of the kernel-based approximation.  
These corrections ensure that the kernel-based approximation gives a consistent high order representation on a bounded domain. 

\subsection{Main idea of the approach}
Now we examine how to make these operators to be high order. We start by observing that directly using these integral operators as the are initially designed fails, and we need to develop a modification to achieve high order.  The approach taken here starts \rev{from} considering the case $k=2$.  Consider:
	\begin{align*}
	\mathcal{D}_{L}[\phi](x)
	= - \frac{1}{\alpha}\phi'(x)- \frac{1}{\alpha^2}\phi''(x) + \frac{1}{\alpha^2}\phi''(b) e^{-\alpha(b-x)} - \frac{1}{\alpha^{3}}I_{L}[\phi'''](x).
	\end{align*}
Instead of adding successively defined terms $\mathcal{D}_{L}^p=\mathcal{D}_{L}[\mathcal{D}_{L}^{p-1}]$, $p\geq 2$, we introduce a new approximation,
	$$ \mtD_{L}[\phi](x) :=\mathcal{D}_{L}[\phi](x) - \frac{1}{\alpha^2}\phi''(b) e^{-\alpha(b-x)} $$
and apply the $\mathcal{D}_L$ operator to this approximation to get,
\begin{align*} 
\mathcal{D}_L[\mtD_{L}[\phi]](x) = \frac{1}{\alpha^2}\phi''(x)+\frac{1}{\alpha^{3}}I_{L}[\phi'''](x)+\frac{1}{\alpha^{3}}I_{L}^2[\phi'''](x).
\end{align*}
We observe that this new compound approximation  gives a  higher order representation for $\phi'(x)$,
\begin{align*}
\mathcal{D}_L[\phi](x) + \mathcal{D}_L[\mtD_{L}[\phi]](x) - \frac{1}{\alpha^2}\phi''(b) e^{-\alpha(b-x)} = - \frac{1}{\alpha}\phi'(x) + \mathcal{O}(\frac{1}{\alpha^3}).
\end{align*}
We generalize this modification to get a higher order in the following.

\subsection{Modification of derivative operators}

Using the general boundary treatments, we can modify the partial sums for the first derivative operators.
\begin{subequations}\label{eq:mod_sum}
\begin{align} 
\phi_{x}^{+}(x) &\approx {\mP}^{L}_{k}[\phi](x)=
-\alpha \left( \mD_{L}[\phi](x)+ \sum\limits_{p=2}^{k}\mtD_{L}^p[\phi](x) \right), \\
\phi_{x}^{-}(x) &\approx{\mP}^{R}_{k}[\phi](x)=
\alpha \left( \mD_{R}[\phi](x)+\sum\limits_{p=2}^{k}\mtD_{R}^p[\phi](x)  \right) \\
\phi_{xx}(x) &\approx{\mP}^{0}_{k}[\phi](x)=
-\alpha^2  \left( \mD_{0}[\phi](x)+\sum\limits_{p=2}^{k}\mtD_{0}^p[\phi](x) \right)
\label{eq:mod_sum_0}
\end{align}
\end{subequations}
where ${\mtD}_{L}$, ${\mtD}_{R}$ and ${\mtD}_{0}$ are given as
\begin{subequations}\label{eq:expression}
	\begin{align}
	& \left\{\begin{array}{ll}\label{eq:expressionL}
    \mttD_{L}[\phi](x) =\mD_{L}[\phi](x) +  \sum_{m=2}^{k} c_{1,m} \left(\frac{1}{\alpha}\right)^{m}\partial_{x}^{m}\phi(b) e^{-\alpha(b-x)},\\
	\mttD_{L}^p[\phi](x) =\mtD_{L}^p[\phi](x) +  \sum_{m=p}^{k} c_{p,m} \left(\frac{1}{\alpha}\right)^{m}\partial_{x}^{m}\phi(b) e^{-\alpha(b-x)}, \quad p\geq 2,\\
	\displaystyle \mtD_{L}^p[\phi](x) =\mathcal{D}_{L}[\mttD_{L}^{p-1}][\phi](x), \quad p\geq 2,\\
    c_{1,m} = -1, c_{p,m} = -\sum\limits_{i=p-1}^{m-1} c_{p-1,i},
	\end{array}
	\right.\\
    & \left\{\begin{array}{ll}\label{eq:expressionR}
    \mttD_{R}[\phi](x) =\mD_{R}[\phi](x) +  \sum_{m=2}^{k} c_{1,m} \left(-\frac{1}{\alpha}\right)^{m}\partial_{x}^{m}\phi(b) e^{-\alpha(x-a)},\\
	\mttD_{R}^p[\phi](x) =\mtD_{R}^p[\phi](x) +  \sum_{m=p}^{k} c_{p,m} \left(-\frac{1}{\alpha}\right)^{m}\partial_{x}^{m}\phi(a) e^{-\alpha(x-a)},\quad p\geq 2,\\
	\displaystyle \mtD_{R}^p[\phi](x) =\mathcal{D}_{R}[\mttD_{R}^{p-1}][\phi](x), \quad p\geq 2,\\
    c_{1,m} = -1, c_{p,m} = -\sum\limits_{i=p-1}^{m-1} c_{p-1,i},
	\end{array}
	\right.\\
    & \left\{\begin{array}{ll}\label{eq:expression0}
	\mttD_{0}^p[\phi](x) = \mtD_{0}^p[\phi](x) + \frac{1}{2}\sum\limits_{m=p+1}^{k} c_{p,m-1} \left( \left(\frac{1}{\alpha}\right)^{2m} \partial_{x}^{2m}\phi(a) + \left(-\frac{1}{\alpha}\right)^{2m+1} \partial_{x}^{2m+1}\phi(a)\right) e^{-\alpha(x-a)}
      \\ 
      \qquad \qquad  \qquad \qquad \quad + \frac{1}{2} \sum\limits_{m=p+1}^{k} c_{p,m-1} \left( \left(\frac{1}{\alpha}\right)^{2m} \partial_{x}^{2m}\phi(b) + \left(\frac{1}{\alpha}\right)^{2m+1} \partial_{x}^{2m+1}\phi(b)\right) e^{-\alpha(b-x)},\\
     \qquad \qquad  \qquad \qquad \qquad \qquad \qquad \qquad \qquad \qquad \qquad \qquad \qquad \qquad \qquad \qquad \qquad \qquad p\geq 1,\\
	\displaystyle \mtD_{0}^p[\phi](x) =\mathcal{D}_{0}[\mttD_{0}^{p-1}][\phi](x), \quad p\geq 2,\\
    c_{1,m} = -1, c_{p,m} = -\sum\limits_{i=p-1}^{m-1} c_{p-1,i},
	\end{array}
	\right.
	\end{align}
\end{subequations}
In developing the modified sums for these operators, we assume that the derivatives of $\phi$ have been constructed at the boundaries, i.e., $\partial_{x}^{m}\phi(a)$ and $\partial_{x}^{m}\phi(b)$, $m\geq1$. 
In practice, for different PDEs, using inverse Lax-Wendroff (ILW) method gives us either odd-order derivatives, or even-order derivatives or all the derivatives of $\phi$; see for example \cite[Section 2.2.3]{christlieb2019kernel}. However, there seems to be missing information in the case of $\mtD_0$. If we have only odd-order or even-order derivatives, the difference of those derivatives is a high-order affect so we can make use of the information we have as a replacement of the missing terms. We will provide more details when we discuss the implementation in Sec. \ref{sec:Numerical Integration Strategies}. 

\subsection{Convergence analysis}
The modified partial sum \eqref{eq:mod_sum} is constructed so that it agrees with the derivative values at the boundary, to preserve consistency with the boundary condition imposed on $\phi$. 
To verify the accuracy of the modified sum for the approximation to $\partial_x$, we start with the following lemma.
\begin{lem}\label{lem:D_til}
For $2 \le p \le k$,
  \begin{align}\label{eqn:DtLAnotherform}
	 \mtD_{L}^{p}[\phi](x) = & \sum_{m=p}^{k}c_{p,m}\left(\frac{1}{\alpha}\right)^{m} \left(\partial^{m}_{x}\phi(x) -\partial_{x}^{m}\phi(b) e^{-\alpha(b-x)} \right) \nonumber\\
 & + (-1)^p \left(\frac{1}{\alpha}\right)^{k+1}  \sum\limits_{m = 1}^{p} |c_{p-m+1,k-m+1}|(I_L)^m[\partial_{x}^{k+1}\phi](x).
\end{align}
\end{lem}

\begin{proof}
      When $p = 2$, \eqref{eqn:DtLAnotherform} is the same as what we get from \eqref{eq:expressionL}. In fact, with \eqref{eqn:DtLAnotherform}, we have
	\begin{align*}
	\mttD_{L}[\phi](x) =\mtD_{L}[\phi](x) -  \sum_{m=2}^{k}\left(\frac{1}{\alpha}\right)^{m}\partial_{x}^{m}\phi(b) e^{-\alpha(b-x)}
	\end{align*}
	and then we can differentiate it,
	\begin{align*}
	\begin{split}
		\partial_{x}\mttD_L[\phi](x) &= \partial_{x} \left( -\sum_{p=1}^{k} \left(\frac{1}{\alpha}\right)^{p}\partial_{x}^{p}\phi(x) -\left(\frac{1}{\alpha}\right)^{k+1}I_{L}[\partial_{x}^{k+1}\phi](x) \right) \\
		&=	-\sum_{p=1}^{k} \frac{1}{\alpha^{p}}\partial_{x}^{p+1}\phi(x)- \frac{1}{\alpha^{k+1}}\partial_{x}I_L[\partial_{x}^{k+1}\phi](x) \\
		&= 	-\sum_{p=1}^{k-1} \frac{1}{\alpha^{p}}\partial_{x}^{p+1}\phi(x)- \frac{1}{\alpha^{k}}I_L[\partial_{x}^{k+1}\phi](x),
	\end{split}
	\end{align*}
	by using Lemma \ref{lem_D}. Repeating the above process, we can find a general form for $0\leq \ell \leq k$,
	\begin{align*}
		\partial_{x}^{\ell}\mttD_L[\phi](x) = -\sum_{p=1+\ell}^{k} \frac{1}{\alpha^{p-\ell}}\partial_{x}^{p}\phi(x)- \frac{1}{\alpha^{k+1-\ell}}I_L[\partial_{x}^{k+1}\phi](x).
	\end{align*}
	Now again from the equation in Lemma \ref{lem_D} with using $k-1$-th sum instead of $k$-th sum, we apply the operator $\mD_L$ to $\mttD_{L}[\phi]$:
        \begin{align*}
	\mtD_{L}^{2}[\phi](x) & =\mathcal{D}_{L}[\mttD_{L}^{1}][\phi](x)  \\
 & = \sum_{p=2}^{k}(p-1)\left(\frac{1}{\alpha}\right)^{p} \left(\partial^{p}_{x}\phi(x) -\partial_{x}^{p}\phi(b) e^{-\alpha(b-x)} \right) \\
 & \quad + \left(\frac{1}{\alpha}\right)^{k+1} I_L[I_L[\partial_{x}^{k+1}\phi]](x) + (k-1) \left(\frac{1}{\alpha}\right)^{k+1}I_L[\partial_{x}^{k+1}\phi](x),
	\end{align*}
 which is same as \eqref{eqn:DtLAnotherform} with $p = 2$. 

 Now we prove $p > 2$ cases by induction. Assume that for $p \ge 2$, \eqref{eqn:DtLAnotherform} is equivalent from the definition of \eqref{eq:expressionL}, then we consider the $p+1$ case. Since
 \begin{align*}
	\mttD_{L}^{p}[\phi](x)  = & \sum_{m=p}^{k}c_{p,m}\left(\frac{1}{\alpha}\right)^{m} \partial^{m}_{x}\phi(x)+(-1)^p \left(\frac{1}{\alpha}\right)^{k+1}  \sum\limits_{m = 1}^{p} |c_{p-m+1,k-m+1}|(I_L)^m[\partial_{x}^{k+1}\phi](x), 
	\end{align*}
 and then we can differentiate it,
 \begin{align*}
	\begin{split}
		\partial_{x}\mttD_{L}^{p}[\phi](x) &= \partial_{x} \left( \sum_{m=p}^{k}c_{p,m}\left(\frac{1}{\alpha}\right)^{m} \partial^{m}_{x}\phi(x)+(-1)^p \left(\frac{1}{\alpha}\right)^{k+1}  \sum\limits_{m = 1}^{p} |c_{p-m+1,k-m+1}|(I_L)^m[\partial_{x}^{k+1}\phi](x) \right) \\
		&= \sum_{m=p}^{k}c_{p,m}\left(\frac{1}{\alpha}\right)^{m} \partial^{m+1}_{x}\phi(x)+(-1)^p \left(\frac{1}{\alpha}\right)^{k+1}  \sum\limits_{m = 1}^{p} |c_{p-m+1,k-m+1}|\partial_{x}(I_L)^m[\partial_{x}^{k+1}\phi](x)  \\
		&= \sum_{m=p}^{k}c_{p,m}\left(\frac{1}{\alpha}\right)^{m} \partial^{m+1}_{x}\phi(x)+ (-1)^p \left(\frac{1}{\alpha}\right)^{k} |c_{p,m}|(I_L[\partial_{x}^{k+1}\phi](x) - \partial_{x}^{k+1}\phi) \\
        & \quad + (-1)^p \left(\frac{1}{\alpha}\right)^{k}  \sum\limits_{m = 2}^{p} |c_{p-m+1,k-m+1}|\left((I_L)^m[\partial_{x}^{k+1}\phi](x) - (I_L)^{m-1}[\partial_{x}^{k+1}\phi](x)\right) \\ 
        &= \sum_{m=p}^{k-1}c_{p,m}\left(\frac{1}{\alpha}\right)^{m} \partial^{m+1}_{x}\phi(x) + (-1)^p \left(\frac{1}{\alpha}\right)^{k} (I_L)^p [\partial_{x}^{k+1}\phi](x) \\
        & \quad + (-1)^p \left(\frac{1}{\alpha}\right)^{k}  \sum\limits_{m = 1}^{p-1} \left(|c_{p-m+1,k-m+1} + c_{p-m,k-m}|\right)(I_L)^m[\partial_{x}^{k+1}\phi](x) \\
        &= \sum_{m=p}^{k-1}c_{p,m}\left(\frac{1}{\alpha}\right)^{m} \partial^{m+1}_{x}\phi(x) + (-1)^p \left(\frac{1}{\alpha}\right)^{k}  \sum\limits_{m = 1}^{p} |c_{p-m+1,k-m}|(I_L)^m[\partial_{x}^{k+1}\phi](x),
	\end{split}
	\end{align*}
    where we note that if $p$ is even, $|c_{p,m}| = c_{p,m}$ and if $p$ is odd, $|c_{p,m}| = -c_{p,m}$. Note that $c_{p,m-1} = c_{p,m} + C(p-1,m-1)$ and repeating the above process, we can find a general form for $0 \leq \ell \leq k+1-p$,we have   
    \begin{align*}
	\begin{split}
		\partial_{x}^l \mttD_{L}^{p}[\phi](x) &= \sum_{m=p}^{k-\ell}c_{p,m}\left(\frac{1}{\alpha}\right)^{m} \partial^{m+\ell}_{x}\phi(x) \\  
         & \quad + (-1)^p \left(\frac{1}{\alpha}\right)^{k-\ell+1}  \sum\limits_{m = 1}^{p} |c_{p-m+1,k-m-\ell+1}|(I_L)^m[\partial_{x}^{k+1}\phi](x) \\
         & = \sum_{m=p+\ell}^{k}c_{p,m-\ell}\left(\frac{1}{\alpha}\right)^{m-\ell} \partial^{m}_{x}\phi(x) \\  
         & \quad + (-1)^p \left(\frac{1}{\alpha}\right)^{k-\ell+1}  \sum\limits_{m = 1}^{p} |c_{p-m+1,k-m-\ell+1}|(I_L)^m[\partial_{x}^{k+1}\phi](x)
	\end{split}
	\end{align*}
 	Now again from the equation in Lemma \ref{lem_D}, using $k-p$ -th sum instead of $k$-th sum, we apply the operator $\mD_L$ to $\mttD_{L}^p[\phi]$:
        \begin{align*}
	\mtD_{L}^{p+1}[\phi](x) & =\mathcal{D}_{L}[\mttD_{L}^{p}][\phi](x)  \\
 & = -\sum_{\ell=1}^{k-p} \left(\frac{1}{\alpha}\right)^{\ell} \left( \partial_{x}^{\ell} \mttD_{L}^{p}[\phi](x) - \partial_{x}^{\ell}\mttD_{L}^{p}[\phi](b) e^{-\alpha(b-x)} \right) -\left(\frac{1}{\alpha}\right)^{k-p+1}I_{L}[\partial_{x}^{k-p+1} \mttD_{L}^{\ell}[\phi]](x) \\
 & = -\sum_{\ell=1}^{k-p} \left(\frac{1}{\alpha}\right)^{\ell} \sum_{m=p+\ell}^{k}c_{p,m-\ell}\left(\frac{1}{\alpha}\right)^{m-\ell}  \left(\partial^{m}_{x}\phi(x) -  \partial^{m}_{x}\phi(b) e^{-\alpha(b-x)}\right) \\
 & \quad + (-1)^{p+1} \left(\frac{1}{\alpha}\right)^{k+1} \sum_{\ell=1}^{k-p} \sum\limits_{m = 1}^{p} |c_{p-m+1,k-m-\ell+1}|(I_L)^m[\partial_{x}^{k+1}\phi](x)  \\
 & \quad + (-1)^{p+1} \left(\frac{1}{\alpha}\right)^{k+1} \sum\limits_{m = 1}^{p} |c_{p-m+1,p-m}|(I_L)^{m+1}[\partial_{x}^{k+1}\phi](x) \\
 & = \sum_{m=p+1}^{k} c_{p+1,m}  \left(\frac{1}{\alpha}\right)^{m} \left(\partial^{m}_{x}\phi(x) -  \partial^{m}_{x}\phi(b) e^{-\alpha(b-x)}\right) + (-1)^{p+1} \left(\frac{1}{\alpha}\right)^{k+1} (I_L)^{p+1}[\partial_{x}^{k+1}\phi](x)  \\
 & \quad + (-1)^{p+1} \left(\frac{1}{\alpha}\right)^{k+1} \sum_{m=1}^{p} \left|\sum_{i = p-m+1}^{k-m} c_{p-m+1,i} \right| (I_L)^{m}[\partial_{x}^{k+1}\phi](x) \\
 & = \sum_{m=p+1}^{k} c_{p+1,m} \left(\frac{1}{\alpha}\right)^{m} \left(\partial^{m}_{x}\phi(x) -  \partial^{m}_{x}\phi(b) e^{-\alpha(b-x)}\right)  \\
 & \quad + (-1)^{p+1} \left(\frac{1}{\alpha}\right)^{k+1} \left( \sum_{m=1}^{p} |c_{p-m+2,k-m+1}| (I_L)^{m}[\partial_{x}^{k+1}\phi](x) + (I_L)^{p+1}[\partial_{x}^{k+1}\phi](x) \right)\\
 & = \sum_{m=p+1}^{k} c_{p+1,m} \left(\frac{1}{\alpha}\right)^{m} \left(\partial^{m}_{x}\phi(x) -  \partial^{m}_{x}\phi(b) e^{-\alpha(b-x)}\right)  \\
 & \quad + (-1)^{p+1} \left(\frac{1}{\alpha}\right)^{k+1} \left( \sum_{m=1}^{p+1} |c_{p-m+2,k-m+1}| (I_L)^{m}[\partial_{x}^{k+1}\phi](x) \right),
	\end{align*}
 which proves \eqref{eqn:DtLAnotherform}.
\end{proof}

\begin{thm}
 Suppose $\phi\in\mathcal{C}^{k+1}[a,b]$, $k=1, \, 2,\, 3$. Then, the modified partial sums \eqref{eq:mod_sum} satisfy
	\begin{align}
		 \|\partial_{x}\phi(x)-{\mP}^{*}_{k}[\phi](x)\|_{\infty}\leq C \left(\frac{1}{\alpha}\right)^{k} \|\partial_{x}^{k+1}\phi(x)\|_{\infty}
	\end{align}
	where $*$ indicates $L$ and $R$ operators.
\end{thm}
\begin{proof}
	It is obvious for the $k=1$ case from the partial sums \eqref{eq:sum} and Lemma \ref{lem_D}. We here consider left-sided first derivative approximations, ${\mP}^{L}_{k}[\phi]$, for $k\geq 2$ cases.

By Lemma \ref{lem:D_til}, using \eqref{eqn:DtLAnotherform} we see that 
\begin{align*}
 \sum\limits_{p=1}^{k}\mtD_{L}^p[\phi](x) & = -\frac{1}{\alpha}\partial_{x}\phi(x) - \sum_{p=2}^{k} \left(\frac{1}{\alpha}\right)^{p} \left(\partial_{x}^{p}\phi(x) -\partial_{x}^{p}\phi(b) e^{-\alpha(b-x)} \right)- \left(\frac{1}{\alpha}\right)^{k+1}I_{L}[\partial_{x}^{k+1}\phi](x)  \\ 
&  \quad + \sum\limits_{p=2}^{k} \sum_{m=p}^{k}c_{p,m}\left(\frac{1}{\alpha}\right)^{m} \left(\partial^{m}_{x}\phi(x) -\partial_{x}^{m}\phi(b) e^{-\alpha(b-x)} \right) \\
& \quad +  \sum\limits_{p=2}^{k} (-1)^p \left(\frac{1}{\alpha}\right)^{k+1}  \sum\limits_{m = 1}^{p} |c_{p-m+1,k-m+1}|(I_L)^m[\partial_{x}^{k+1}\phi](x) \\
& = -\frac{1}{\alpha}\partial_{x}\phi(x) + \sum_{p=2}^{k} C_{1,p} \left(\frac{1}{\alpha}\right)^{p} \left(\partial_{x}^{p}\phi(x) -\partial_{x}^{p}\phi(b) e^{-\alpha(b-x)} \right)+  \left(\frac{1}{\alpha}\right)^{k+1}  \sum\limits_{p=1}^{k} C_{2,p} (I_L)^p[\partial_{x}^{k+1}\phi](x),
\end{align*}
where $C_{1,p}$ and $C_{2,p}$ are two different constants only depending on $p$. This equation proves the case for $k \ge 2$, \eqref{eqn:DtLAnotherform} is a correct approximation of $\phi_x^+$ and so is \eqref{eq:expressionL}. The proof for \eqref{eq:expressionR} is similar to \eqref{eq:expressionL}.
\end{proof}

Now we will verify the accuracy of the modified sum for the approximation to $\partial_{xx}$, as above, we start with the following lemma.
\begin{lem}\label{lem:D0_til}
For $2 \le p \le k$,
  \begin{align}\label{eqn:Dt0Anotherform}
  \mtD_{0}^{p}[\phi](x) = & \sum_{m=p}^{k}c_{p,m}\left(\frac{1}{\alpha}\right)^{2m} \partial^{2m}_{x}\phi(x) \nonumber \\ 
     & - \frac{1}{2}\sum\limits_{m=p+1}^{k} c_{p,m-1} \left( \left(\frac{1}{\alpha}\right)^{2m} \partial_{x}^{2m}\phi(a) + \left(-\frac{1}{\alpha}\right)^{2m+1} \partial_{x}^{2m+1}\phi(a)\right) e^{-\alpha(x-a)}
      \nonumber \\ 
     & - \frac{1}{2} \sum\limits_{m=p+1}^{k} c_{p,m-1} \left( \left(\frac{1}{\alpha}\right)^{2m} \partial_{x}^{2m}\phi(b) + \left(\frac{1}{\alpha}\right)^{2m+1} \partial_{x}^{2m+1}\phi(b)\right) e^{-\alpha(b-x)}  \nonumber\\
 & + (-1)^p \left(\frac{1}{\alpha}\right)^{2k+2}  \sum\limits_{m = 1}^{p} |c_{p-m+1,k-m+1}|(I_L)^m[\partial_{x}^{2k+2}\phi](x).
\end{align}
\end{lem}

\begin{proof}
  When $p = 2$, \eqref{eqn:Dt0Anotherform} is same as what we get from \eqref{eq:expression0}. In fact, with \eqref{eqn:Dt0Anotherform}, we have
	\begin{align*}
	\mttD_{0}[\phi](x) = \mtD_{0}[\phi](x) & - \frac{1}{2}\sum\limits_{m=2}^{k} \left( \left(\frac{1}{\alpha}\right)^{2m} \partial_{x}^{2m}\phi(a) + \left(-\frac{1}{\alpha}\right)^{2m+1} \partial_{x}^{2m+1}\phi(a)\right) e^{-\alpha(x-a)}
      \\ 
      & - \frac{1}{2} \sum\limits_{m=2}^{k} \left( \left(\frac{1}{\alpha}\right)^{2m} \partial_{x}^{2m}\phi(b) + \left(\frac{1}{\alpha}\right)^{2m+1} \partial_{x}^{2m+1}\phi(b)\right) e^{-\alpha(b-x)}.
	\end{align*}
	Then similar to $\mttD_L[\phi](x)$, we can differentiate it,and by using Lemma \ref{lem_D},
	\begin{align*}
	\begin{split}
		\partial_{xx} \mttD_0[\phi](x) &= \partial_{xx} \left( -\sum\limits_{p=1}^{k} \left(\frac{1}{\alpha}\right)^{2p} \partial_{x}^{2p}\phi(x) - \left(\frac{1}{\alpha}\right)^{2k+2} I_{0}[\partial_{x}^{2k+2}\phi](x) \right) \\
		&=	-\sum\limits_{p=2}^{k} \left(\frac{1}{\alpha}\right)^{2p-2} \partial_{x}^{2p}\phi(x) - \left(\frac{1}{\alpha}\right)^{2k} I_{0}[\partial_{x}^{2k+2}\phi](x).
	\end{split}
	\end{align*}
	Repeating the above process, we can find a general form for $0\leq \ell \leq k$,
	\begin{align*}
		\partial_{x}^{2\ell}\mttD_0[\phi](x) = -\sum\limits_{p=1+\ell}^{k} \left(\frac{1}{\alpha}\right)^{2p-2\ell} \partial_{x}^{2p}\phi(x) - \left(\frac{1}{\alpha}\right)^{2k+2-2\ell} I_{0}[\partial_{x}^{2k+2}\phi](x).
	\end{align*}
We can also derive the following expression for $0\leq \ell \leq k$, which we will need later,
    \begin{align*}
		\partial_{x}^{2\ell+1}\mttD_0[\phi](x) = -\sum\limits_{p=1+\ell}^{k} \left(\frac{1}{\alpha}\right)^{2p-2\ell} \partial_{x}^{2p+1}\phi(x) - \left(\frac{1}{\alpha}\right)^{2k+2-2\ell} \partial_x I_{0}[\partial_{x}^{2k+2}\phi](x),
	\end{align*}
    where we use the property $\partial_x I_{0}[\partial_{x}^{2k+2}\phi](a) = \alpha I_{0}[\partial_{x}^{2k+2}\phi](a)$ and $\partial_x I_{0}[\partial_{x}^{2k+2}\phi](b) = -\alpha I_{0}[\partial_{x}^{2k+2}\phi](b)$.
	Now again from the equation in Lemma \ref{lem_D} with using $k-1$-th sum instead of $k$-th sum, we apply the operator $\mD_0$ to $\mttD_{0}[\phi]$:
        \begin{align*}
	\mtD_{0}^{2}[\phi](x) 
 &=\mathcal{D}_{0}[\mttD_{0}^{1}][\phi](x)  \\
 &= \sum_{m=2}^{k}(m-1)\left(\frac{1}{\alpha}\right)^{2m} \partial^{2m}_{x}\phi(x) \\
 & \quad - \frac{1}{2}\sum\limits_{m=3}^{k} (m-2) \left( \left(\frac{1}{\alpha}\right)^{2m} \partial_{x}^{2m}\phi(a) + \left(-\frac{1}{\alpha}\right)^{2m+1} \partial_{x}^{2m+1}\phi(a)\right) e^{-\alpha(x-a)}  \\ 
     & \quad - \frac{1}{2} \sum\limits_{m=3}^{k} (m-2) \left( \left(\frac{1}{\alpha}\right)^{2m} \partial_{x}^{2m}\phi(b) + \left(\frac{1}{\alpha}\right)^{2m+1} \partial_{x}^{2m+1}\phi(b)\right) e^{-\alpha(b-x)} \\
 & \quad + \left(\frac{1}{\alpha}\right)^{2k+2} I_0[I_0[\partial_{x}^{2k+2}\phi]](x) + (k-1) \left(\frac{1}{\alpha}\right)^{2k+2}I_0[\partial_{x}^{2k+2}\phi](x),
	\end{align*}
 which is the same as \eqref{eqn:DtLAnotherform} with $p = 2$. 
 Now we prove $p > 2$ cases by induction. Assume that for $p \ge 2$, and note that \eqref{eqn:DtLAnotherform} is equivalent from the definition of \eqref{eq:expressionL}, and now consider $p+1$ case, 
 \begin{align*}
	\mttD_{0}^{p}[\phi](x)  = & \sum_{m=p}^{k}c_{p,m}\left(\frac{1}{\alpha}\right)^{2m} \partial^{2m}_{x}\phi(x) + (-1)^p \left(\frac{1}{\alpha}\right)^{k+1}  \sum\limits_{m = 1}^{p} |c_{p-m+1,k-m+1}|(I_0)^m[\partial_{x}^{k+1}\phi](x),
	\end{align*}
 and differentiating it gives,
 \begin{align*}
	\begin{split}
		\partial_{xx}\mttD_{0}^{p}[\phi](x) &= \partial_{xx} \left( \sum_{m=p}^{k}c_{p,m}\left(\frac{1}{\alpha}\right)^{2m} \partial^{2m}_{x}\phi(x) + (-1)^p \left(\frac{1}{\alpha}\right)^{2k+2}  \sum\limits_{m = 1}^{p} |c_{p-m+1,k-m+1}|(I_0)^m[\partial_{x}^{2k+2}\phi](x) \right) \\
		&= \sum_{m=p}^{k}c_{p,m}\left(\frac{1}{\alpha}\right)^{2m} \partial^{2m+2}_{x}\phi(x) + (-1)^p \left(\frac{1}{\alpha}\right)^{2k+2}  \sum\limits_{m = 1}^{p} |c_{p-m+1,k-m+1}| \partial_{xx}(I_0)^m[\partial_{x}^{2k+2}\phi](x)  \\
		&= \sum_{m=p}^{k}c_{p,m}\left(\frac{1}{\alpha}\right)^{2m} \partial^{2m+2}_{x}\phi(x) + (-1)^p \left(\frac{1}{\alpha}\right)^{2k} |c_{p,2k}|(I_0[\partial_{x}^{2k+2}\phi](x) - \partial_{x}^{2k+2}\phi) \\
        & \quad + (-1)^p \left(\frac{1}{\alpha}\right)^{2k}  \sum\limits_{m = 2}^{p} |c_{p-m+1,k-m+1}|\left((I_0)^m[\partial_{x}^{2k+2}\phi](x) - (I_0)^{m-1}[\partial_{x}^{2k+2}\phi](x)\right) \\ 
        &= \sum_{m=p}^{k-1}c_{p,m}\left(\frac{1}{\alpha}\right)^{2m} \partial^{2m+2}_{x}\phi(x) + (-1)^p \left(\frac{1}{\alpha}\right)^{2k} (I_0)^p [\partial_{x}^{2k+2}\phi](x) \\
        & \quad + (-1)^p \left(\frac{1}{\alpha}\right)^{2k}  \sum\limits_{m = 1}^{p-1} \left(|c_{p-m+1,k-m+1} + c_{p-m,k-m}|\right)(I_0)^m[\partial_{x}^{2k+2}\phi](x) \\
        &= \sum_{m=p}^{k-1}c_{p,m}\left(\frac{1}{\alpha}\right)^{2m} \partial^{2m+2}_{x}\phi(x) + (-1)^p \left(\frac{1}{\alpha}\right)^{2k}  \sum\limits_{m = 1}^{p} |c_{p-m+1,k-m}|(I_0)^m[\partial_{x}^{2k+2}\phi](x),
	\end{split}
	\end{align*}
    where note that if $p$ is even, $|c_{p,m}| = c_{p,m}$ and if $p$ is odd, $|c_{p,m}| = -c_{p,m}$. Note that $c_{p,m-1} = c_{p,m} + c_{p-1,m-1}$ and repeating the above process, we can find a general form for $0 \leq \ell \leq k+1-p$, leading to  
    \begin{align*}
	\begin{split}
		\partial_{x}^{2\ell} \mttD_{0}^{p}[\phi](x) &= \sum_{m=p}^{k-\ell}c_{p,m}\left(\frac{1}{\alpha}\right)^{2m} \partial^{2m+2\ell}_{x}\phi(x) \\  
         & \quad + (-1)^p \left(\frac{1}{\alpha}\right)^{2k-2\ell+2}  \sum\limits_{m = 1}^{p} |c_{p-m+1,k-m-\ell+1}|(I_0)^m[\partial_{x}^{2k+2}\phi](x) \\
         & = \sum_{m=p+\ell}^{k}c_{p,m-\ell}\left(\frac{1}{\alpha}\right)^{2m-2\ell} \partial^{2m}_{x}\phi(x) \\  
         & \quad + (-1)^p \left(\frac{1}{\alpha}\right)^{2k-2\ell+2}  \sum\limits_{m = 1}^{p} |c_{p-m+1,k-m-\ell+1}|(I_0)^m[\partial_{x}^{2k+2}\phi](x).
	\end{split}
	\end{align*}
    Additionally, for $0 \leq \ell \leq k+1-p$, we have,
    \begin{align*}
	\begin{split}
		\partial_{x}^{2\ell+1} \mttD_{0}^{p}[\phi](x) &= 
         \sum_{m=p+\ell}^{k}c_{p,m-\ell}\left(\frac{1}{\alpha}\right)^{2m-2\ell} \partial^{2m+1}_{x}\phi(x) \\  
         & \quad + (-1)^p \left(\frac{1}{\alpha}\right)^{2k-2\ell+2}  \sum\limits_{m = 1}^{p} |c_{p-m+1,k-m-\ell+1}|\partial_x(I_0)^m[\partial_{x}^{2k+2}\phi](x),
	\end{split}
	\end{align*}
    where we have used the properties  $\partial_x(I_0)^m[\partial_{x}^{2k+2}\phi](a) = \alpha (I_0)^m[\partial_{x}^{2k+2}\phi](a)$ and $\partial_x(I_0)^m[\partial_{x}^{2k+2}\phi](b) = -\alpha (I_0)^m[\partial_{x}^{2k+2}\phi](b)$.
 	Now again from the equation in Lemma \ref{lem_D} with using $k-p$ -th sum instead of $k$-th sum, we apply the operator $\mD_0$ to $\mttD_{0}^p[\phi]$:
        \begin{align*}
	\mtD_{0}^{p+1}[\phi](x) & =\mathcal{D}_{0}[\mttD_{0}^{p}][\phi](x)  \\
 & = -\sum\limits_{\ell=1}^{k-p} \left(\frac{1}{\alpha}\right)^{2\ell} \partial_{x}^{2\ell}\mttD_{0}^{p}[\phi](x) - \left(\frac{1}{\alpha}\right)^{2k+2-2p} I_{0}[\partial_{x}^{2k+2-2p}\mttD_{0}^{p}[\phi]](x) \\
    & \quad \, + \frac{1}{2}\sum\limits_{\ell=4}^{2k+1-2p} \left(-\frac{1}{\alpha}\right)^{\ell} \partial_{x}^{\ell}\mttD_{0}^{p}[\phi](a) e^{-\alpha(x-a)}  + \frac{1}{2}\sum\limits_{\ell=4}^{2k+1-2p} \left(\frac{1}{\alpha}\right)^{\ell} \partial_{x}^{\ell}\mttD_{0}^{p}[\phi](b) e^{-\alpha(b-x)} \\
 & = -\sum_{\ell=1}^{k-p} \sum_{m=p+\ell}^{k}c_{p,m-\ell}\left(\frac{1}{\alpha}\right)^{2m} \partial^{2m}_{x}\phi(x) \\
 & \quad + \sum_{\ell=1}^{k-p}\sum\limits_{m = 1}^{p} (-1)^{p+1}  \left(\frac{1}{\alpha}\right)^{2k+2}   |c_{p-m+1,k-m-\ell+1}|(I_0)^m[\partial_{x}^{2k+2}\phi](x) \\
 & \quad + (-1)^{p+1} \left(\frac{1}{\alpha}\right)^{2k+2} \sum\limits_{m = 1}^{p} |c_{p-m+1,p-m}|(I_L)^{m+1}[\partial_{x}^{2k+2}\phi](x) \\
 & \quad + \frac{1}{2}\sum_{\ell=2}^{k-p} \sum_{m=p+\ell}^{k}c_{p,m-\ell}\left(\frac{1}{\alpha}\right)^{2m} \left( \partial^{2m}_{x}\phi(a)e^{-\alpha(x-a)} + \partial^{2m}_{x}\phi(b)e^{-\alpha(b-x)} \right)\\
 & \quad + \frac{1}{2}\sum_{\ell=2}^{k-p} \sum_{m=p+\ell}^{k}c_{p,m-\ell}\left(\frac{1}{\alpha}\right)^{2m} \left( - \partial^{2m}_{x}\phi(a)e^{-\alpha(x-a)} + \partial^{2m}_{x}\phi(b)e^{-\alpha(b-x)} \right)\\
 & = \sum_{m=p+1}^{k} c_{p+1,m} \left(\frac{1}{\alpha}\right)^{2m} \partial^{2m}_{x}\phi(x) + (-1)^{p+1} \left(\frac{1}{\alpha}\right)^{2k+2} (I_0)^{p+1}[\partial_{x}^{2k+2}\phi](x)  \\
& \quad + (-1)^{p+1} \left(\frac{1}{\alpha}\right)^{2k+2} \sum_{m=1}^{p} \left|\sum_{i = p-m+1}^{k-m} c_{p-m+1,i} \right| (I_0)^{m}[\partial_{x}^{2k+2}\phi](x) \\
 & \quad - \frac{1}{2} \sum_{m=p+2}^{k} c_{p+1,m-1} \left(\frac{1}{\alpha}\right)^{2m} \left( \partial^{2m}_{x}\phi(a)e^{-\alpha(x-a)} + \partial^{2m}_{x}\phi(b)e^{-\alpha(b-x)} \right)\\
 & \quad - \frac{1}{2}\sum_{m=p+2}^{k} c_{p+1,m-1} \left(\frac{1}{\alpha}\right)^{2m} \left( - \partial^{2m}_{x}\phi(a)e^{-\alpha(x-a)} + \partial^{2m}_{x}\phi(b)e^{-\alpha(b-x)} \right)\\
 & = \sum_{m=p+1}^{k} c_{p+1,m} \left(\frac{1}{\alpha}\right)^{2m} \partial^{2m}_{x}\phi(x)  \\
& \quad + (-1)^{p+1} \left(\frac{1}{\alpha}\right)^{2k+2} \left( \sum_{m=1}^{p+1} |c_{p-m+2,k-m+1}| (I_0)^{m}[\partial_{x}^{2k+2}\phi](x) \right)\\
 & \quad - \frac{1}{2} \sum_{m=p+2}^{k} c_{p+1,m-1} \left(\frac{1}{\alpha}\right)^{2m} \left( \partial^{2m}_{x}\phi(a)e^{-\alpha(x-a)} + \partial^{2m}_{x}\phi(b)e^{-\alpha(b-x)} \right)\\
 & \quad - \frac{1}{2}\sum_{m=p+2}^{k} c_{p+1,m-1} \left(\frac{1}{\alpha}\right)^{2m} \left( - \partial^{2m}_{x}\phi(a)e^{-\alpha(x-a)} + \partial^{2m}_{x}\phi(b)e^{-\alpha(b-x)} \right)\\
\end{align*}
 which proves \eqref{eqn:Dt0Anotherform}.
\end{proof}

\begin{thm}
 Suppose $\phi\in\mathcal{C}^{2k+2}[a,b]$, $k=1, \, 2,\, 3$. Then, the modified partial sums \eqref{eq:mod_sum_0} satisfy
	\begin{align}
		 \|\partial_{xx}\phi(x)-{\mP}^{0}_{k}[\phi](x)\|_{\infty}\leq C \left(\frac{1}{\alpha}\right)^{2k} \|\partial_{x}^{2k+2}\phi(x)\|_{\infty}.
	\end{align}
\end{thm}

\begin{proof}
It is obvious for the $k=1$ case from the partial sums \eqref{eq:D0sum} and Lemma \ref{lem_D}. We now consider second derivative approximations, ${\mP}^{0}_{k}[\phi]$, for $k \geq 2$ cases. 

By Lemma \ref{lem:D0_til}, using the format \eqref{eqn:Dt0Anotherform}, we see that 
\begin{align*}
 \sum\limits_{p=1}^{k}\mtD_{0}^p[\phi](x) & = -\frac{1}{\alpha}^2\partial_{xx}\phi(x) -\sum\limits_{p=2}^{k} \left(\frac{1}{\alpha}\right)^{2p} \partial_{x}^{2p}\phi(x) - \left(\frac{1}{\alpha}\right)^{2k+2} I_{0}[\partial_{x}^{2k+2}\phi](x) \\
      & \quad + \frac{1}{2}\sum\limits_{p=4}^{2k+1} \left(-\frac{1}{\alpha}\right)^{p} \partial_{x}^{p}\phi(a) e^{-\alpha(x-a)}
     + \frac{1}{2}\sum\limits_{p=4}^{2k+1} \left(\frac{1}{\alpha}\right)^{p} \partial_{x}^{p}\phi(b) e^{-\alpha(b-x)} \\
    &  \quad + \sum_{p=2}^{k} \sum_{m=p}^{k}c_{p,m}\left(\frac{1}{\alpha}\right)^{2m} \partial^{2m}_{x}\phi(x) \nonumber \\ 
     & - \frac{1}{2}\sum_{p=2}^{k} \sum\limits_{m=p+1}^{k} c_{p,m-1} \left( \left(\frac{1}{\alpha}\right)^{2m} \partial_{x}^{2m}\phi(a) + \left(-\frac{1}{\alpha}\right)^{2m+1} \partial_{x}^{2m+1}\phi(a)\right) e^{-\alpha(x-a)}
      \nonumber \\ 
     & - \frac{1}{2} \sum_{p=2}^{k} \sum\limits_{m=p+1}^{k} c_{p,m-1} \left( \left(\frac{1}{\alpha}\right)^{2m} \partial_{x}^{2m}\phi(b) + \left(\frac{1}{\alpha}\right)^{2m+1} \partial_{x}^{2m+1}\phi(b)\right) e^{-\alpha(b-x)}  \nonumber\\
 & + (-1)^p \left(\frac{1}{\alpha}\right)^{2k+2} \sum_{p=2}^{k} \sum\limits_{m = 1}^{p} |c_{p-m+1,k-m+1}|(I_0)^m[\partial_{x}^{2k+2}\phi](x) \\
 & = -\frac{1}{\alpha}^2\partial_{xx}\phi(x) + \sum\limits_{p=2}^{k} C_{1,p} \left(\frac{1}{\alpha}\right)^{2p} \partial_{x}^{2p}\phi(x) + \left(\frac{1}{\alpha}\right)^{2k+2} \sum_{m=1}^{k} C_{2,p}(I_0)^m[\partial_{x}^{2k+2}\phi](x) \\
     & + \sum_{m=2}^{k} C_{3,p} \left( \left(\frac{1}{\alpha}\right)^{2m} \partial_{x}^{2m}\phi(a) + \left(-\frac{1}{\alpha}\right)^{2m+1} \partial_{x}^{2m+1}\phi(a)\right) e^{-\alpha(x-a)} \\ 
     & + \sum_{m=2}^{k} C_{3,p} \left( \left(\frac{1}{\alpha}\right)^{2m} \partial_{x}^{2m}\phi(b) + \left(\frac{1}{\alpha}\right)^{2m+1} \partial_{x}^{2m+1}\phi(b)\right) e^{-\alpha(b-x)}
\end{align*}
where $C_{1,p}$, $C_{2,p}$ and $C_{3,p}$ are three constants only depending on $p$. This equation proves the case for $k \ge 2$, \eqref{eqn:Dt0Anotherform} is a correct approximation of $\phi_{xx}$ and so is \eqref{eq:expression0}. 

\end{proof}

\section{Numerical Integration Strategies}\label{sec:Numerical Integration Strategies}

In this section, we present the approximation for 
the differential operators \eqref{eq:D_ops}
$\mathcal{D}_{L}$ and $\mathcal{D}_{R}$ and \eqref{eq:D0_ops} $\mathcal{D}_0$.
To simplify matters, let us consider the case where the spatial domain is limited to one dimension, say  $[a,b]$, and partition it with
\begin{align*}
a=x_{0}<x_{1}<\cdots<x_{N-1} <x_{N} =b,
\end{align*}
where $\Delta x_i=x_{i+1}-x_{i}$ for $i=0,\cdots,N-1$.
In what follows, we will address both the cases of uniform and nonuniform grid spacing. 

From the equations \eqref{def:Int_ops} and \eqref{def:Int0_ops}, by dividing the integral range into cell units for $i=0,\cdots,N-1$,
\begin{align*}
	I_{L,i} := \alpha \int_{x_i}^{x_{i+1}} e^{-\alpha (s-x)}v(s) \, ds 
 \qquad \text{and}\qquad
 I_{R,i} := \alpha \int_{x_i}^{x_{i+1}} e^{-\alpha (x-s)}v(s) \, ds, 
\end{align*}
we have the recursive relations
\begin{align}\label{eq:recursive}
    \begin{split}	
	I_{L,i} &= e^{-\alpha\Delta x_{i+1}} I_{L,i+1} + J_{L,i},\,\quad i=0,\ldots,N-1, \\
	I_{R,i} &= e^{-\alpha\Delta x_{i}} I_{R,i-1} + J_{R,i},\qquad i=1,\ldots,N,
    \end{split}
\end{align}
where $I_{L,0} = 0$ and $I_{R,N} = 0$
with
\begin{align}\label{eq:JLR}
J_{L,i} =  \alpha \int_{x_{i}}^{x_{i+1}} e^{-\alpha (s-x_{i})} v(s) ds
\qquad \text{and}\qquad
J_{R,i} =  \alpha \int_{x_{i-1}}^{x_{i}} e^{-\alpha (x_{i}-s)} v(s) ds.
\end{align}

The remaining task is to approximate $J_{L,i} $ and $J_{R,i}$. On a uniformly distributed spatial domain, 
a high order methodologies to approximate $J_{L,i} $ and $J_{R,i}$ is suggested in \cite{christlieb2020, christlieb2019kernel} making use of classical WENO integration \cite{jiang2000weighted}. 
On a nonuniformly distributed domain, a high order methodology is suggested in \cite{christlieb2020kernel}, where an exponential based form WENO \cite{HKYY} on a uniform domain under a coordinate transformation is introduced. In this paper, we address the ENO-type approach rather than WENO for the nonuniform grid case. In Section 5.1, we will briefly review the development method for the uniform case, and in Section 5.2, an integration method utilizing the ENO \cite{shu1988eno, shu1989eno} concept will be introduced in the context of nonuniform boundaries.

\subsection{WENO on uniform grids}\label{subsec:weno}

To approximate $J_{L,i}$ and $J_{R,i}$ with fifth order of accuracy,
we choose a six-point stencil $\SS:=\SS(i)=\{x_{i-2},\ldots, x_{i+3}\}$ and three sub-stencils $S_0, S_1, S_2$ defined by $S_r:=S_{r}(i) =\{ x_{i-2+r},\ldots, x_{i+1+r}\}$ for $r = 0, 1, 2$.
We focus on computing $J^L_{i}$ here and note that the approximation for $J^R_{i}$ can be computed in a mirror symmetric way.

We first obtain the interpolating polynomial $p$ of degree less than or equal to 5 to $v$ on the big stencil $\SS$ and 
interpolants $p_r$ to $v$ of degree at most 3 using the substencils $S_r$ for $r=0,1,2$. Then, we define $J^{L}_{i}$ and $J^{L}_{i,r}$ approximating $J_{L,i}$ from $p$ and $p_r$, respectively:
\begin{align}\label{eq:gloJR}
    J^{L}_{i} :=\alpha \int_{x_{i}}^{x_{i+1}}e^{-\alpha(x_{i} - s)} p(s) \,ds  = J_{L,i}+ \mO(\Delta x^6),
\end{align}
and 
\begin{align}\label{eq:localJR}
    J^{L}_{i,r} :=\alpha \int_{x_{i}}^{x_{i+1}}e^{-\alpha(x_{i} - s)} p_r(s)  \,ds   =J_{L,i}+ \mO(\Delta x^4), 
\end{align}
for $r=0,1,2$, where the function $v$ is smooth enough.

Secondly, WENO strategies address the linear and nonlinear weights to combine local approximations $J^{L}_{i,r}$ on the sub-stencils for the construction of the final approximation to $J_{L,i}$.
With linear weights $d_k$, we want the combination of local solutions $J^{L}_{i,r}$ is consistent with the global solution $J^{L}_{i}$, i.e.,
\begin{equation}\label{weno:dr}
	J^{L}_{i} =\sum_{r=0}^{2}d_{r}J^{L}_{i,r}	
\end{equation}
such that a high order method is obtained. 
For the essentially non-oscillatory property, we introduce the nonlinear weights $\omega_r$ by constructing smoothness indicators, which are factors that assess the characteristics of functions within local stencils. For the details, we recommend that the reader refer to \cite{jiang1996efficient, jiang2000weighted}.
We note that both linear and nonlinear coefficients $d_{r}$ and $\omega_r$ are defined to satisfy the partition of unity, i.e., the sum of each should be equal to unity.

\subsection{ENO on nonuniform grids}\label{subsec:ENO}

When the domain is a bit complicated, for example, a circular domain, our scheme can still be applied to solve corresponding equations, however  a numerical integration strategy for a non-uniform mesh needs to be adopted. In this case we have chosen to make use of a  ``nonuniform" fourth order ENO method to do the spatial discretization which is efficient and accurate enough for smooth test problems and a good choice if we are  considering a non-smooth problem.  It should be noted that one can use other appropriate spatial discretizations, such as ``nonuniform" WENO \cite{marti2024efficient}. Let us take a circular domain as an example to illustrate how we do the implementation. 

\begin{figure}[ht]
    \centering
    \includegraphics[width=0.4\textwidth]{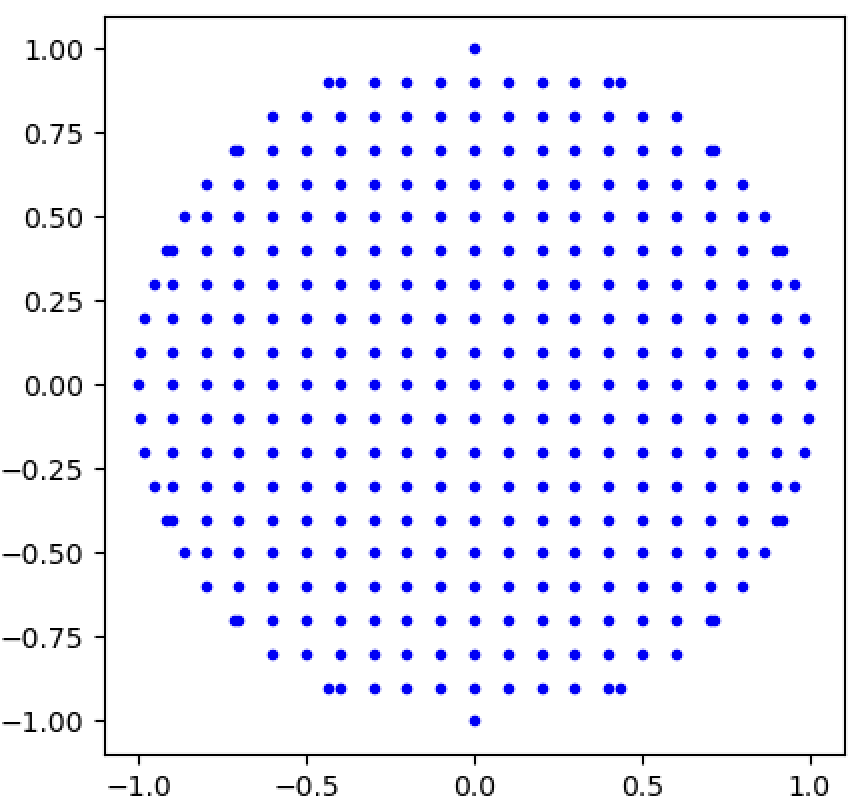}
    \includegraphics[width=0.4\textwidth]{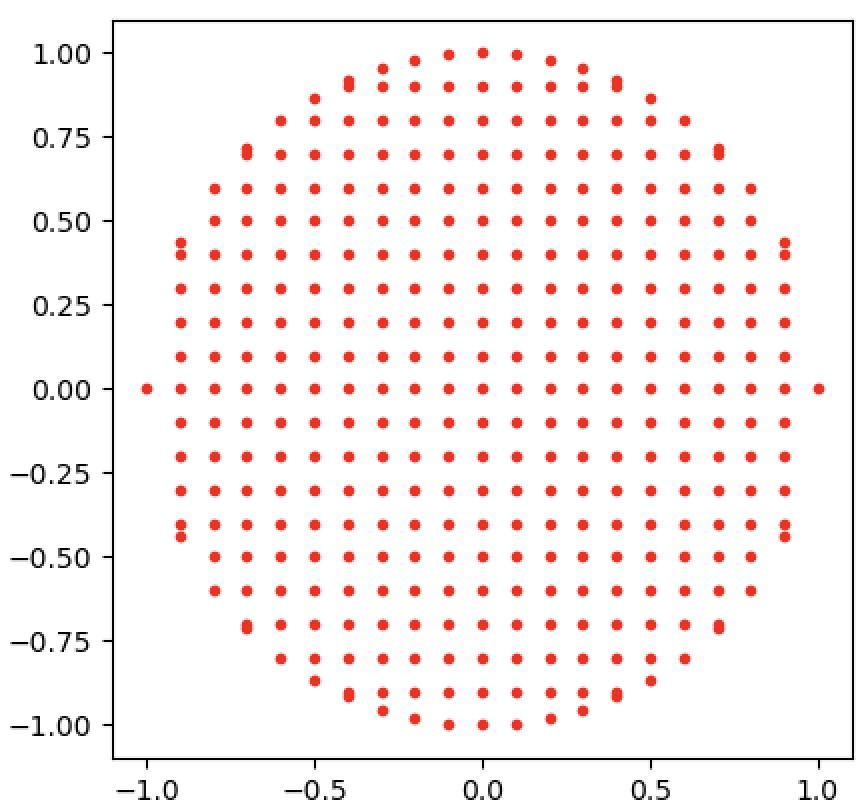}
    \caption{Circular mesh along $x$-axis (blue) and $y$-axis (red).}
    \label{fig:nonunifrom mesh}
\end{figure}

We firstly discretize the domain by embedding it in a regular Cartesian mesh of say $N_y$ along the $x$-axis and $N_x$ along the $y$-axis, and additionally incorporating the termination points of each line, which will lie on the boundary (boundary points); see Figure \ref{fig:nonunifrom mesh}. Thus in this example,  for each line, the interior grids are uniform and thus we only need to care about the nonuniform mesh at the boundary. In Figure \ref{fig: nonuniform ENO method}, let ``$h$" be the length of uniform cells and ``$a$" is one of the boundary points of this line, with the cell length $\lambda h$. 
In order to get the integration over the cells around the boundary, such as $J_1^L$ and $J_0^R$, we first extend the boundary with several cells with length $h$  and then make use of extrapolation to approximate the function values at those points. For example, if we extend two additional points with fourth order extrapolation, then the coefficients are given by the following:
\begin{align}\label{eq:nonuniformENO_extension}
    v_{-2} & = \frac{(5+\lambda)(4+\lambda)}{\lambda(1+\lambda)}v_0 - \frac{(5+\lambda)(4+\lambda)(3+\lambda)}{3\lambda}v_1 + \frac{(5+\lambda)(4+\lambda)(2+\lambda)}{(1+\lambda)}v_2 
    \nonumber \\ & \qquad 
    - (5+\lambda)(3+\lambda)v_3 + \frac{(4+\lambda)(2+\lambda)}{3}v_4 \\ 
    v_{-1} & = \frac{4+\lambda}{\lambda}v_0 - \frac{(4+\lambda)(3+\lambda)(2+\lambda)}{6\lambda}v_1 + \frac{(4+\lambda)(3+\lambda)}{2}v_2 
    \nonumber \\ & \qquad 
    - \frac{(4+\lambda)(1+\lambda)}{2}v_3 + \frac{(2+\lambda)(1+\lambda)}{6}v_4,
\end{align}
where $v_0, v_1, v_2, v_3, v_4, v_{-1}, v_{-2}$ are the function values at grid points $x_0=a,x_1,x_2,x_3,x_4$ and the extended two points outside the domain $x_{-1},x_{-2}$, respectively.
With this extrapolation, we approximate $J_L$ and $J_R$ using ENO method, whose construction falls a similar process to that in WENO in the uniform case shown in \eqref{eq:localJR}, but with the mesh length parameter $\lambda$.
There are three different stencils used in \eqref{eq:localJR}, and while any of them can be chosen, it is important to use the same stencil for both $J_L$ and $J_R$ in the implementation.  

\begin{figure}[h]
    \centering
    \includegraphics[width=0.6\textwidth]{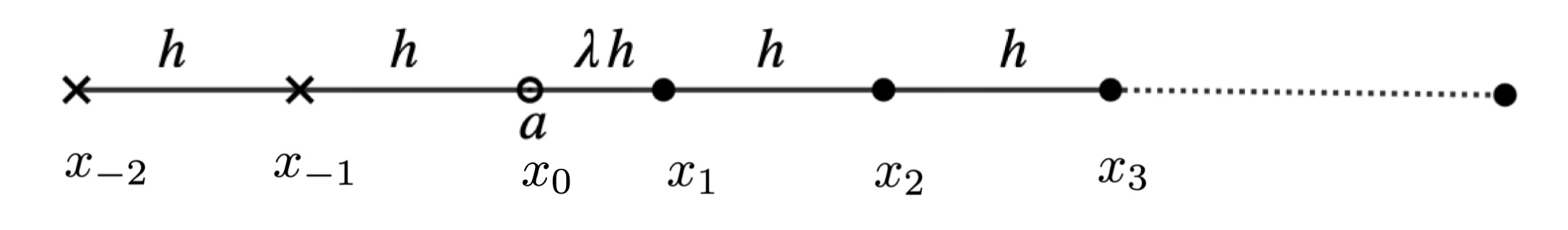}
    \caption{Nonuniform third order ENO method.}
    \label{fig: nonuniform ENO method}
\end{figure}

\subsection{Time integration using SSP-RK methods}\label{sec:rev-time}

To achieve high-order temporal accuracy, we employ the strong-stability-preserving Runge–Kutta (SSP-RK) framework for time integration. Specifically, for all numerical examples presented in Section \ref{sec:result}, we adopt the third-order, three-stage SSP-RK method. At each intermediate Runge–Kutta stage, the spatial kernel-based operator is applied to the evolving solution. This integration strategy ensures that the combined scheme maintains both high-order accuracy and desirable stability properties.

This approach follows standard practice in the method-of-lines setting and has been shown to be effective in previous kernel-based schemes \cite{gottlieb2001strong}. We note that in earlier sections, backward Euler was used for deriving the analytic form of the kernel operators, but all numerical results use SSP-RK unless otherwise stated. The general form of the SSP-RK(3,3) scheme is given by:
\begin{align*}
u^{(1)} &= u^n + \Delta t \, \mathcal{L}(u^n), \\
u^{(2)} &= \frac{3}{4} u^n + \frac{1}{4} \left( u^{(1)} + \Delta t \, \mathcal{L}(u^{(1)}) \right), \\
u^{n+1} &= \frac{1}{3} u^n + \frac{2}{3} \left( u^{(2)} + \Delta t \, \mathcal{L}(u^{(2)}) \right),
\end{align*}
where \( \mathcal{L}(u) \) denotes the spatial discretization using the kernel-based operator.

\subsection{Algorithm}

In this section, as an example, we outline the algorithm for solving one-dimensional linear diffusion equation with Dirichlet boundary condition; see Example \ref{ex:diffusive} in Sec. \ref{sec:result} and Algorithm \ref{Algorithm1}. \rev{The overall algorithm is as follows:}
   \RestyleAlgo{ruled}
    \begin{algorithm}[h!]
        \caption{Outline of the MOL$^T$ - general operator scheme for solving one-dimensional linear diffusion equation with Dirichlet boundary condition. }\label{Algorithm1}

            \textbf{1.} Compute $\alpha := \frac{\beta}{q \Delta t}$, where $\beta$ is chosen based on \rev{Table \ref{tab:beta_max} and to be more specific here, we use $\beta_{2,k,max}$ for different $k-$order scheme.} 
            
    \If{the mesh is uniform}{
        \textbf{2.} Approximate the integrals $J_L$ and $J_R$ using WENO quadrature; see Section \ref{subsec:weno}\;
    }
    \Else{
        \textbf{2.} Approximate the integrals $J_L$ and $J_R$ using ENO quadrature; see Section \ref{subsec:ENO}\;
    }
    \textbf{3.} Using the previously computed approximations $J_L$ and $J_R$, we apply the recurrence relations \eqref{eq:recursive} to obtain the convolution integrals. Then, with Dirichlet boundary condition \eqref{Dir_bdy_AB0}, we can compute $\mD_0$ by \eqref{eq:D0_ops}, which is similar to \cite{christlieb2020,christlieb2019kernel}. 

    \textbf{4.} To get $\mtD_0$ in \eqref{eq:expression0}, we use ILW to get all the even-order derivatives needed. To be more specific, we use the following formula to get $\mttD_0$:
    \begin{align*}
        \mttD_{0}^p[\phi](x) = \mtD_{0}^p[\phi](x) + \sum\limits_{m=p+1}^{k} c_{p,m-1} \left( \left(\frac{1}{\alpha}\right)^{2m} \partial_{x}^{2m}\phi(a) - \mu  \left(\frac{1}{\alpha}\right)^{2m} \partial_{x}^{2m}\phi(b)\right) e^{-\alpha(x-a)}  \\
        +  \sum\limits_{m=p+1}^{k} c_{p,m-1} \left( \left(\frac{1}{\alpha}\right)^{2m} \partial_{x}^{2m}\phi(b) - \mu \left(\frac{1}{\alpha}\right)^{2m}\partial_{x}^{2m}\phi(a)\right) e^{-\alpha(b-x)},  \quad p \ge 1, 
    \end{align*}
    and then with \eqref{eq:expression0} to get $\mtD_0$.

    \textbf{5.} Repeat Step \textbf{2--4} $k-$times to approximate the second derivative $u_{xx}$ with $k$ partial sums of $\mtD$ and when $k \ge 2$, we only need to apply WENO or ENO quadrature to the first terms; see more details in \cite[Section 5]{christlieb2020}. This leads to the partial sum approximations of $u_{xx}$ with \eqref{eq:mod_sum_0}.  

    \textbf{6.} Let the partial sum approximation from Step \textbf{5} be the right-hand side of 
    $u_t = f(u,t)$, then apply $k-$order SSP RK method to solve the ODE to get the final solution.         
    \end{algorithm}

\rev{The following remarks are meant to charily Algorithm \ref{Algorithm1}:}
\begin{itemize}
   \item \textbf{Remark about nonlinear problems:} \rev{Although Algorithm \ref{Algorithm1} only shows the process for a linear problem, it also applies to problems such as
   \[
        u_t  = g(u)_{xx}, 
    \]
        where $g$ can be linear or nonlinear operators. We let the operators $\mtD_0$ apply to $g(u^n)$ to get $\mtD_0[g(u^n)](x)$ with the wave speed $q$ chosen by $q:= \max|g'(u)|$. For the case of first derivative including Hamilton-Jacobian case, the process is similar; see more details in \cite{christlieb2019kernel, christlieb2020}. }
    \item \textbf{Remark about step 1:} \rev{We first need to choose $\alpha$, which is defined by $\alpha := \frac{\beta}{q \Delta t}$. $q$ is coming from the problem and the parameter $\beta$ is determined by the stability analysis. To be self-contained, we present Theorem \ref{thm:stability_beta} and Table \ref{tab:beta_max} to show how we choose beta; see more details including the proof of Theorem \ref{thm:stability_beta} in \cite{christlieb2020}.}
\begin{table}[h]
\centering
\begin{tabular}{c|c|c|c}
\hline
$k$ & $\beta_{1,k,\max}$ & $\beta_{2,k,\max}$ & $\beta_{k,\max}$\\
\hline
1 & 2 & 2 & 1\\
2 & 1 & 1 & 0.5 \\
3 & 1.243 & 0.8375 & 0.4167\\
\hline
\end{tabular}
\caption{ $\beta_{\max}$ in Theorem \ref{thm:stability_beta} for $k=1,2,3$}
\label{tab:beta_max}
\end{table}
\rev{
\begin{thm}\label{thm:stability_beta}
(\cite[Table 1]{christlieb2020}) \begin{itemize}
    \item[(a)] \textit{For the linear advection equation $u_t + c u_x = 0$ with periodic boundary conditions, there exists constant $\beta_{1,k,\max} > 0$ for $k=1,2$, such that the scheme is A-stable provided $0 < \beta \leq \beta_{1,k,\max}$.}
    \item[(b)] \textit{For the linear diffusion equation $u_t = q u_{xx}$ with $q>0$ and periodic boundary conditions, there exists constant $\beta_{2,k,\max} > 0$ for $k=1,2,3$, such that the scheme is A-stable provided $0 < \beta \leq \beta_{2,k,\max}$.}
\end{itemize}
\textit{The constants $\beta_{1,k,\max}$ and $\beta_{2,k,\max}$ are summarized in Table~\ref{tab:beta_max}.}
\end{thm}}
\rev{Furthermore, for the advection-diffusion case, the appropriate choice of $\beta$ is listed under the column $\beta_{k,\max}$ in Table~\ref{tab:beta_max}.}
\item \textbf{Remark about step 3:} \rev{For the periodic case, the operators do not need additional information for boundary derivatives. For other cases like Dirichlet boundary conditions, inverse Lax-Wendroff (ILW) method is used to get boundary derivatives (if applicable) or analytic boundary derivatives are used for the purpose of verifying the operators. Take the following parabolic equations as an example,  
\[
u_t  = g(u)_{xx}. 
\]
To get the even derivatives required for $\mtD_0[g(u)](x)$ in \eqref{eq:expression0}, we use boundary derivatives to get: 
\[
\partial_x^{2m} g(u)|_{\partial\Omega} =   \partial_t^{m} u|_{\partial\Omega}.
\]
Similar process applies to (nonlinear) advection problems and Hamilton-Jacobian problems; see more details in \cite{christlieb2020kernel, christlieb2019kernel,christlieb2020}}
\end{itemize}
\section{Numerical results}\label{sec:result}
In this section, we present several numerical results of the proposed scheme.  While we have demonstrated the method for a range of non-linear problems in the past \cite{christlieb2019kernel,christlieb2020}, here we wanted to focus on the theory as it relates boundary conditions for a range of operators.  thus we focus on linear problems here.   The linear test problems considered include the: 1D diffusion equation; 1D wave equation; 1D convection-diffusion; 2D convection; 2D diffusion on a non-uniform mesh; 2D wave equation on a circular domain.  These are considered for a range of boundary conditions, including: Periodic, Dirichlet; time-independent Dirichlet; Inflow; and Neumann.  We present data that supports the theory in this paper, demonstrating that the methods can achieve the desired order of accuracy.  

\paragraph{1D examples} We start with our 1D examples.  
\begin{exa}\label{ex:diffusive}
    We test the accuracy of the scheme for the one-dimensional linear diffusion problem:
\begin{align} 
\left\{\begin{array}{ll}
u_{t} = q \,u_{xx}, & a < x < b,\\
u(x,0)= f(x).
\end{array}
\right.
\end{align}
Here, $q\geq0$ is a given constant. We can test this problem with different boundary conditions and initial conditions.

(1) Periodic boundary condition: Let $a = -\pi, b = \pi$, $f(x) = \sin(x)$, with periodic boundary imposed and then the exact solution $u^e(x,t) = \sin(x-t)$.  


(2) Time-independent Dirichlet boundary condition: Let $a = 0, b = 1$ $u(0,t) = 0, u(1,t) = 1$ and then the exact solution
$u^{e}(x,t)= \sum\limits_{n = 1}^{\infty}b_n \sin(n \pi x)e^{-n^2\pi^2t} + x$, where $\sum\limits_{n = 1}^{\infty}b_n \sin(n \pi x) + x = f(x)$. We test the case when $q = 1$ and $f(x) = \sin(\pi x) + x$ at final time $T = 2$. 


(3) Inflow Dirichlet boundary condition: Let $a = -1, b = 1$, $u(-1,t) = - e^{-\frac{1}{2}-\frac{1}{4}t}$ and $u(1,t) = e^{\frac{1}{2}-\frac{1}{4}t}$, $f(x) = (\sin(\pi x) + x)e^{0.5x}$, and then the exact solution
$u^{e}(x,t)=e^{-\frac{1}{4}t - \pi^2 t +\frac{1}{2}x}\sin(\pi x) + xe^{\frac{1}{2}x-\frac{1}{4}t}$. The numerical test is done at final time $T = 0.5$. 


(4) Neumann boundary condition: Let $a = 0, b = 1$, $u_x(0,t) = \pi e^{-\pi^2 t} + e^{t}$ and $u_x(1,t) = -\pi e^{-\pi^2 t} + e^{1+t}$, $f(x) = \sin(\pi x) + e^{x}$, and then the exact solution
$u^{e}(x,t)=\sin(\pi x)e^{-\pi^2 t} + e^{x+t}$. The numerical test is done at final time $T = 0.5$. 

\end{exa} 

In Table \ref{tab:ex_diff_periodic}, Table \ref{tab:ex_diff_dir}, Table \ref{tab:ex_diff_inflow}, and Table \ref{tab:ex_diff_Neu}, we summarize the convergence study for the periodic boundary condition, time-independent Dirichlet boundary condition, inflow Dirichlet boundary condition, and Neumann boundary condition, respectively. The $L_\infty$-errors and corresponding orders of accuracy are also presented. It is observed that the use of the $k$-th partial sum yields $k$-th order
accuracy, thereby verifying the analysis presented in Sect.\ref{sec:newbdyterm}. In particular, for the case of periodic boundary conditions and time-independent Dirichlet boundary conditions, second order accuracy is observed for the case $k = 1$. The superconvergence for the first order scheme, with $k = 1$, is expected since the first order scheme for the general operator is exact same as the previous kernel-base operators and thus, when applied to the linear problem, is equivalent to the second order Crank-Nicolson scheme \cite{christlieb2016weno}.  Moreover, the scheme allows
for large CFL numbers due to its unconditional stability.

\begin{table}[htbp]
	\centering
\vspace{2mm}
	\begin{small}
		\begin{tabular}{|c|c|cc|cc|cc|}
			\hline
			\multirow{2}{*}{CFL} &  \multirow{2}{*}{$N$} & \multicolumn{2}{c|}{$k=1$.} & \multicolumn{2}{c|}{$k=2$.} & \multicolumn{2}{c|}{$k=3$.}\\
			\cline{3-8}
			& &  error &   order  &  error &  order  &  error  & order  \\\hline
			\multirow{5}{*}{0.5} 
            &  20  & 3.97e-03 &   --& 5.44e-02 &   --& 1.34e-02 &   --\\
			&  40  & 1.01e-03 & 1.972& 1.41e-02 & 1.951& 1.79e-03 & 2.905\\
			&  80  & 2.56e-04 & 1.981& 3.58e-03 & 1.973& 2.30e-04 & 2.965\\
			& 160  & 6.41e-05 & 1.998& 8.97e-04 & 1.998 & 2.88e-05 & 2.994\\
			& 320  & 1.60e-05 & 1.999& 2.24e-04 & 1.999& 3.61e-06 & 2.998\\
            \hline   
			\multirow{5}{*}{1} 
            &  20  & 1.53e-02 &   -- &1.92e-01 &   -- & 9.36e-02 &   --\\
			&  40  & 3.98e-03 & 1.943 & 5.43e-02 & 1.823& 1.36e-02 & 2.781 \\
            &  80  &1.01e-03 & 1.977 & 1.41e-02 & 1.946& 1.80e-03 & 2.920\\
            & 160  & 2.56e-04 & 1.981 & 3.58e-03 & 1.976&2.30e-04 & 2.968 \\
            & 320  & 6.41e-05 & 1.999 & 8.97e-04 & 1.998& 2.88e-05 & 2.996\\
            \hline    
			\multirow{5}{*}{2}  
            &  20  & 5.85e-02 &   --& 5.21e-01 &   --& 4.19e-01 &   --\\
			&  40  & 1.54e-02 & 1.927& 1.92e-01 & 1.438&  9.36e-02 & 2.160\\
            &  80  & 3.99e-03 & 1.949 &5.45e-02 & 1.819 & 1.37e-02 & 2.778\\
            & 160  & 1.01e-03 & 1.977&1.41e-02 & 1.950 & 1.80e-03 & 2.924\\
             & 320 & 2.56e-04 & 1.982&3.58e-03 & 1.976 &  2.30e-04 & 2.968\\
            \hline   
		\end{tabular}
	\end{small}	
	\caption{\label{tab:ex_diff_periodic} $L_{\infty}$-errors and orders of accuracy for Example \ref{ex:diffusive} (1) with periodic boundary conditions at $T=2$. }
\end{table}

\begin{table}[htbp]
	\centering
\vspace{2mm}
	\begin{small}
		\begin{tabular}{|c|c|cc|cc|cc|}
			\hline
			\multirow{2}{*}{CFL} &  \multirow{2}{*}{$N$} & \multicolumn{2}{c|}{$k=1$.} & \multicolumn{2}{c|}{$k=2$.} & \multicolumn{2}{c|}{$k=3$.}\\
			\cline{3-8}
			& &  error &   order  &  error &  order  &  error  & order  \\\hline
			\multirow{5}{*}{0.5} 
			&  40  &  1.79e-04 & --  &     5.43e-04 & -- &  4.53e-04 & -- \\
			&  80  &  4.50e-05 & 1.996  &     1.45e-04 & 1.908 &  8.21e-05 & 2.464 \\
			& 160  &  1.13e-05 & 1.999  &     3.76e-05 & 1.943 &  1.27e-05 & 2.696 \\
			& 320  &  2.81e-06 & 2.000 &     9.61e-06 & 1.968 &  1.77e-06 & 2.838 \\
			& 640  &  7.04e-07 & 1.998 &     2.43e-06 & 1.984 &  2.35e-07 & 2.914 \\
            & 1280 &  1.79e-07 & 1.977  &     5.91e-07 & 2.041 &  2.66e-08 & 3.143 \\
            \hline   
			\multirow{5}{*}{1} 
			&  40  &  7.10e-04 & --  &  7.80e-03 & -- &  2.00e-03 & --\\
            &  80  &  1.79e-04 & 1.985  & 2.01e-03 & 1.957 &  4.53e-04 & 2.143\\
            & 160  &  4.50e-05 & 1.996  &  5.43e-04 & 1.888  & 8.21e-05 & 2.464 \\
            & 320  &  1.13e-05 & 1.999  &  1.45e-04 & 1.908  &  1.27e-05 & 2.696\\
            & 640  &  2.81e-06 & 1.999  &  3.76e-05 & 1.943 & 1.77e-06 & 2.839\\
            & 1280 &  7.07e-07 & 1.993  &  9.61e-06 & 1.969 &  2.41e-07 & 2.878\\
            \hline    
			\multirow{5}{*}{2}  
			&  40  &  2.70e-03 &   --&  3.58e-02 &   --  & 6.50e-03 & -- \\
            &  80  &   7.10e-04 & 1.929 &  7.80e-03 & 2.198  & 2.00e-03 & 1.701 \\
            & 160  &   1.79e-04 & 1.985 &  2.01e-03 & 1.957  &4.53e-04 & 2.143 \\
             & 320  &  4.50e-05 & 1.996  &  5.43e-04 & 1.888  & 8.21e-05 & 2.464 \\
            & 640  &  1.13e-05 & 1.999 &  1.45e-04 & 1.908  & 1.27e-05 & 2.696 \\
            & 1280 &  2.83e-06 & 1.991 &   3.76e-05 & 1.942  & 1.76e-06 & 2.851 \\
            \hline   
		\end{tabular}
	\end{small}	
	\caption{\label{tab:ex_diff_dir} $L_{\infty}$-errors and orders of accuracy for Example \ref{ex:diffusive} (2) with non-homogeneous {\bf time-independent} Dirichlet boundary conditions at $T=2$. }
\end{table}

\begin{table}[htbp] 
	\centering
\vspace{2mm}
	\begin{small}
		\begin{tabular}{|c|c|cc|cc|cc|}
			\hline
			\multirow{2}{*}{CFL} &  \multirow{2}{*}{$N$} & \multicolumn{2}{c|}{$k=1$.} & \multicolumn{2}{c|}{$k=2$.} & \multicolumn{2}{c|}{$k=3$.}\\
			\cline{3-8}
			& &  error &   order  &  error &  order  &  error  & order  \\\hline
			\multirow{5}{*}{0.5}  
            &  20  &  2.20e-02 & --  & 1.88e-03 & -- & 4.61e-04 & -- \\
			&  40  &  1.10e-02 & 0.997  & 5.05e-04 & 1.898 & 8.33e-05 & 2.469 \\
			&  80  &  5.50e-03 & 0.998  &   1.34e-04 & 1.912  &  1.28e-05 & 2.697  \\
			& 160  &  2.75e-03 & 0.999  &  3.49e-05 & 1.944 & 1.79e-06 & 2.839  \\
			& 320  &  1.38e-03 & 1.000  &   8.92e-06 & 1.968 & 2.38e-07 & 2.917 \\
			& 640  &  6.89e-04 & 1.000  &  2.25e-06 & 1.987 &  3.06e-08 & 2.958\\
            \hline
			\multirow{5}{*}{1} 
            &  20  &  4.37e-02 & --  & 7.41e-03 & -- & 2.05e-03 & -- \\
			&  40  &  2.20e-02 & 0.994  &   1.88e-03 & 1.977& 4.61e-04 & 2.155 \\
            &  80  &1.10e-02 & 0.997 & 5.05e-04 & 1.898 & 8.33e-05 & 2.469 \\
            & 160  &  5.50e-03 & 0.998 &   1.34e-04 & 1.912 & 1.28e-05 & 2.697 \\
            & 320  &  2.75e-03 & 0.999  &   3.49e-05 & 1.944 & 1.79e-06 & 2.842 \\
            & 640  &  1.38e-03 & 1.000 &  8.89e-06 & 1.972  & 2.68e-07 & 2.741 \\ 
            \hline
			\multirow{5}{*}{2}  
            &  20  &  8.67e-02 & --  &3.47e-02 & --  & 6.86e-03 & -- \\
			&  40  &  4.37e-02 & 0.988  &7.41e-03 & 2.225  &2.05e-03 & 1.741 \\
			&  80  &  2.20e-02 & 0.994  & 1.88e-03 & 1.977     &  4.61e-04 & 2.155  \\
			& 160  &  1.10e-02 & 0.997  &  5.05e-04 & 1.898 & 8.33e-05 & 2.469 \\
			& 320  &  5.50e-03 & 0.998  &   1.34e-04 & 1.912 & 1.28e-05 & 2.697 \\
			& 640  &  2.75e-03 & 0.999  &  3.49e-05 & 1.942 &  1.72e-06 & 2.901\\
            \hline 
		\end{tabular}
	\end{small}	
	\caption{\label{tab:ex_diff_inflow} $L_{\infty}$-errors and orders of accuracy for Example \ref{ex:diffusive} (3) with non-homogeneous inflow Dirichlet boundary conditions at $T=0.5$. }
\end{table}

\begin{table}[htbp]
	\centering
\vspace{2mm}
	\begin{small}
		\begin{tabular}{|c|c|cc|cc|cc|}
			\hline
			\multirow{2}{*}{CFL} &  \multirow{2}{*}{$N$} & \multicolumn{2}{c|}{$k=1$.} & \multicolumn{2}{c|}{$k=2$.} & \multicolumn{2}{c|}{$k=3$.}\\
			\cline{3-8}
			& &  error &   order  &  error &  order  &  error  & order  \\\hline
			\multirow{5}{*}{0.5} 
            & 40 &  7.50e-02 & -- & 2.21e-03 & -- & 7.00e-03 & --\\
            & 80 &  3.62e-02 & 1.051 & 5.58e-04 & 1.987 &8.76e-04 & 2.999 \\
            & 160 & 1.78e-02 & 1.026  & 1.41e-04 & 1.989 & 1.10e-04 & 2.998\\
            & 320 & 8.81e-03 & 1.013  & 3.54e-05 & 1.991 & 1.37e-05 & 2.997\\ 
            & 640 &  4.39e-03 & 1.006 & 8.94e-06 & 1.984 & 1.72e-06 & 2.999 \\
            \hline   
			\multirow{5}{*}{1} 
            & 40 & 1.61e-01 & -- & 9.36e-03 & -- & 5.59e-02 & --\\
            & 80 & 7.50e-02 & 1.100 & 2.35e-03 & 1.994  & 6.99e-03 & 3.000 \\
            & 160 & 3.62e-02 & 1.051 & 5.92e-04 & 1.989 & 8.75e-04 & 2.999 \\
            & 320 & 1.78e-02 & 1.026 & 1.49e-04 & 1.990 & 1.10e-04 & 2.998 \\
            & 640 & 8.81e-03 & 1.013 & 3.73e-05 & 1.997 & 1.36e-05 & 3.006 \\
            \hline    
			\multirow{5}{*}{2}  
            & 40 & 3.68e-01 & -- &  3.81e-02 & -- & 4.47e-01 &   --\\
            & 80 & 1.61e-01 & 1.194 &  9.43e-03 & 2.015  & 5.59e-02 & 3.000\\
            & 160 & 7.50e-02 & 1.100 &  2.37e-03 & 1.994  &6.99e-03 & 3.000 \\
            & 320 & 3.62e-02 & 1.051 &  5.96e-04 & 1.990  & 8.75e-04 & 2.999\\
            & 640 & 1.78e-02 & 1.026 &  1.50e-04 & 1.991  & 1.10e-04 & 2.998\\
            \hline   
		\end{tabular}
	\end{small}	
	\caption{\label{tab:ex_diff_Neu} $L_{\infty}$-errors and orders of accuracy for Example \ref{ex:diffusive} (4) with Neumann boundary conditions at $T=0.5$. }
\end{table}

\begin{exa}\label{ex:1d-wave}
    We test the accuracy of the scheme for the one-dimensional wave problem:
\begin{align} 
\left\{\begin{array}{ll}
u_{tt} = q \,u_{xx}, & a < x < b,\\
u(x,0)= f(x), \\
u_t(x,0) = g(x).
\end{array}
\right.
\end{align}
Here, $q \ge 0$ is a given constant. We test this problem with different boundary conditions and initial conditions.

(1) Dirichlet boundary condition: Let $a = \frac{1}{2}$, $b=\frac{3}{2}$, $f(x) = x$,  $g(x) = -\pi \sin(\pi x)$, $q = 1$ and the boundary condition $$u(\frac{1}{2},t) = \frac{1}{2} + \cos((t+\frac{1}{2})\pi),\quad u(\frac{3}{2},t) = \frac{3}{2} + \cos((t+\frac{3}{2})\pi).$$ Then the exact solution is $$u^e = x + \frac{1}{2}(\cos (\pi(x+t)) - \cos (\pi(x-t))).$$

(2) Periodic boundary condition: Let $a = 0$, $b= 2$, $f(x) = \cos (\pi x)$,  $g(x) = -\pi \sin(\pi x)$, $q = 1$ and the problem holds periodic boundary condition, then the exact solution is $$u^e = \cos (\pi(x+t)).$$

In Table \ref{tab:1d-wave_dir} and Table \ref{tab:1d-wave_per}, we present the convergence study, including the 
$L_\infty$-errors and the corresponding orders of accuracy for the Dirichlet and periodic boundary conditions, respectively. The numerical results showed that $k$-th order accuracy is achieved for the $k$-th partial sum approximation to the operator, even with a large CFL. 
\end{exa}

\begin{table}[htbp]
	\centering
\vspace{2mm}
	\begin{small}
		\begin{tabular}{|c|c|cc|cc|cc|}
			\hline
			\multirow{2}{*}{CFL} &  \multirow{2}{*}{$N$} & \multicolumn{2}{c|}{$k=1$.} & \multicolumn{2}{c|}{$k=2$.} & \multicolumn{2}{c|}{$k=3$.}\\
			\cline{3-8}
			& &  error &   order  &  error &  order  &  error  & order  \\\hline
			\multirow{5}{*}{0.5} 
            & 20 &     1.59e-01 &   --         &  3.16e-03 &   --    &  3.27e-04 &   --  \\
            & 40 &     7.89e-02 & 1.007         &  5.90e-04 & 2.423    &   3.04e-05 & 3.423 \\
            & 80 &     3.94e-02 & 1.004        &   1.12e-04 & 2.394    &   2.20e-06 & 3.790 \\
            & 160 &    1.97e-02 & 1.002         &  2.52e-05 & 2.152     &  1.72e-07 & 3.681  \\
            & 320 &    9.82e-03 & 1.001         &   6.31e-06 & 2.000    &  1.71e-08 & 3.324  \\ 
            & 640 &    4.91e-03 & 1.000         &   1.58e-06 & 2.000    &  1.80e-09 & 3.248  \\
            \hline   
			\multirow{5}{*}{1} 
            & 20 &  3.20e-01 &   -- &     1.32e-02 &   --      &2.14e-03 &   --       \\
            & 40 &  1.59e-01 & 1.012 &     3.18e-03 & 2.052       &3.27e-04 & 2.712    \\
            & 80 &  7.89e-02 & 1.007 &     5.90e-04 & 2.430       &3.05e-05 & 3.422        \\
            & 160 & 3.94e-02 & 1.004  &    1.12e-04 & 2.395       &2.20e-06 & 3.791        \\
            & 320 & 1.97e-02 & 1.002  &    2.52e-05 & 2.152       &1.73e-07 & 3.667        \\
            & 640 & 9.82e-03 & 1.001  &    6.31e-06 & 2.000       &1.65e-08 & 3.392        \\
            \hline    
			\multirow{5}{*}{2}  
            & 20 &  6.87e-01 &   --  & 3.96e-02 &   --   &  8.32e-03 &   --\\
            & 40 &  3.20e-01 & 1.103  &  1.32e-02 & 1.580    & 2.14e-03 & 1.957\\
            & 80 &  1.59e-01 & 1.012  &  3.18e-03 & 2.057   & 3.27e-04 & 2.712 \\
            & 160 &  7.89e-02 & 1.007  & 5.90e-04 & 2.430   &  3.05e-05 & 3.422\\
            & 320 &  3.94e-02 & 1.004  & 1.12e-04 & 2.395   &  2.21e-06 & 3.789\\
            & 640 &  1.97e-02 & 1.002  & 2.52e-05 & 2.152   &  2.04e-07 & 3.436\\
            \hline   
		\end{tabular}
	\end{small}	
	\caption{\label{tab:1d-wave_dir} $L_{\infty}$-errors and orders of accuracy for Example \ref{ex:1d-wave} with Dirichlet boundary conditions at $T = 1$. }
\end{table}

\begin{table}[htbp]
	\centering
\vspace{2mm}
	\begin{small}
		\begin{tabular}{|c|c|cc|cc|cc|}
			\hline
			\multirow{2}{*}{CFL} &  \multirow{2}{*}{$N$} & \multicolumn{2}{c|}{$k=1$.} & \multicolumn{2}{c|}{$k=2$.} & \multicolumn{2}{c|}{$k=3$.}\\
			\cline{3-8}
			& &  error &   order  &  error &  order  &  error  & order  \\\hline
			\multirow{5}{*}{0.5} 
            & 20 &    4.83e-01 &   --&  1.75e-01 &   -- &    9.08e-02 &   -- \\
            & 40 &     2.36e-01 & 1.035 &  6.03e-02 & 1.537  &2.08e-02 & 2.128      \\
            & 80 &     1.16e-01 & 1.018 &   1.83e-02 & 1.718  & 3.75e-03 & 2.469     \\
            & 160 &    5.78e-02 & 1.010 &  5.11e-03 & 1.841   & 5.76e-04 & 2.703     \\
            & 320 &    2.88e-02 & 1.005 &  1.36e-03 & 1.915   &  8.04e-05 & 2.843    \\ 
            & 640 &    1.44e-02 & 1.000  &    3.50e-04   &  1.958 &     1.06e-05 & 2.919  \\
            \hline   
			\multirow{5}{*}{1} 
            & 20 & 9.85e-01 &   -- & 4.32e-01 &   --   &  2.98e-01 &   --\\
            & 40 & 4.83e-01 & 1.028 &       1.75e-01 & 1.306   &  9.08e-02 & 1.715 \\
            & 80 & 2.36e-01 & 1.034 &       6.03e-02 & 1.537   &  2.08e-02 & 2.128     \\
            & 160 & 1.16e-01 & 1.019 &      1.83e-02 & 1.718   &  3.75e-03 & 2.469     \\
            & 320 & 5.78e-02 & 1.010 &      5.11e-03 & 1.841   &   5.76e-04 & 2.703    \\
            & 640 &   2.88e-02 & 1.005  &   1.36e-03 &  1.910 &   8.04e-05 & 2.843    \\
            \hline    
			\multirow{5}{*}{2}  
            & 20 &1.92e+00 &   -- &   9.24e-01 &   --&    7.47e-01 &   -- \\
            & 40 & 9.91e-01 & 0.952 &  4.37e-01 & 1.078    &  3.01e-01 & 1.312   \\
            & 80 & 4.83e-01 & 1.037 &  1.75e-01 & 1.318   &   9.09e-02 & 1.727   \\
            & 160 &2.36e-01 & 1.033  & 6.03e-02 & 1.541   &   2.08e-02 & 2.129   \\
            & 320 & 1.16e-01 & 1.019 & 1.83e-02 & 1.718   &   3.75e-03 & 2.469   \\
            & 640 &   5.78e-02 & 1.010   &   5.11e-03 & 1.841    &    5.76e-04 & 2.703   \\
            \hline   
		\end{tabular}
	\end{small}	
	\caption{\label{tab:1d-wave_per} $L_{\infty}$-errors and orders of accuracy for Example \ref{ex:1d-wave} with periodic boundary conditions at $T = 1$. }
\end{table}

\begin{exa}\label{ex:1dAdC}
We test the accuracy of the scheme for the one-dimensional linear convection-diffusion problem
\begin{align}
\left\{\begin{array}{ll}
u_{t}+c\,u_{x}=q\,u_{xx}, & -\frac{1}{2}\leq x\leq \frac{1}{2},\\
u(x,0)=\sin(\pi x)e^{0.5 x},
\end{array}
\right.
\end{align}
with the Dirichilet boundary condition $u(-\frac{1}{2},t) = -e^{-\frac{1}{4} t - \pi^2 t -\frac{1}{4}}$ and $u(\frac{1}{2},t) =e^{-\frac{1}{4} t - \pi^2 t +\frac{1}{4}} $.
Here, $c$ and $q\geq0$ are given constants. Let $c = q = 1$. This problem has the exact solution
$u^{e}(x,t)=e^{-\frac{1}{4} t - \pi^2 t +\frac{1}{2} x}\sin(\pi x).$

In Table \ref{tab:1d_adv}, we summarize the convergence study at final time $T = 0.5$, and the $L_\infty$-errors and the associated orders of accuracy are provided. Even though a large CFL is used, orders of accuracy for schemes with $k = 1$,2, and 3, are shown to be $k^{th}$-order. 
\end{exa}

\begin{table}[htbp]
	\centering
\vspace{2mm}
	\begin{small}
		\begin{tabular}{|c|c|cc|cc|cc|}
			\hline
			\multirow{2}{*}{CFL} &  \multirow{2}{*}{$N$} & \multicolumn{2}{c|}{$k=1$.} & \multicolumn{2}{c|}{$k=2$.} & \multicolumn{2}{c|}{$k=3$.}\\
			\cline{3-8}
			& &  error &   order  &  error &  order  &  error  & order  \\\hline
			\multirow{5}{*}{0.5} 
            & 20 &      5.00e-02 &   --    &  6.70e-03 &   -- &      2.40e-03 &   --   \\
            & 40 &      2.47e-02 & 1.015     & 1.85e-03 & 1.857   &   4.28e-04 & 2.491       \\
            & 80 &      1.23e-02 & 1.008    &  4.81e-04 & 1.943   &   6.46e-05 & 2.726       \\
            & 160 &     6.13e-03 & 1.004     & 1.21e-04 & 1.996    &  8.85e-06 & 2.868        \\
            & 320 &     3.06e-03 & 1.002     & 2.97e-05 & 2.022    &  1.15e-06 & 2.944        \\ 
            & 640 &     1.53e-03 & 1.001     & 7.28e-06 & 2.030    &  1.46e-07 & 2.975        \\
            \hline   
			\multirow{5}{*}{1} 
            & 20 &1.02e-01 &   --   &      2.26e-02 &   --     &  1.09e-02 &   -- \\
            & 40 & 5.00e-02 & 1.029  &     6.71e-03 & 1.755      & 2.41e-03 & 2.182   \\
            & 80 & 2.47e-02 & 1.015  &     1.85e-03 & 1.858      & 4.28e-04 & 2.493       \\
            & 160 & 1.23e-02 & 1.008  &    4.81e-04 & 1.943      & 6.47e-05 & 2.727      \\
            & 320 & 6.13e-03 & 1.004  &    1.21e-04 & 1.996      & 8.86e-06 & 2.868      \\
            & 640 & 3.06e-03 & 1.002  &    2.97e-05 & 2.022      & 1.15e-06 & 2.946      \\
            \hline    
			\multirow{5}{*}{2}  
            & 20 & 2.12e-01 &   -- &   6.98e-02 &   -- &       4.20e-02 &   --    \\
            & 40 & 1.02e-01 & 1.057 &  2.26e-02 & 1.625   &    1.09e-02 & 1.941       \\
            & 80 & 5.00e-02 & 1.029 &  6.71e-03 & 1.755  &     2.41e-03 & 2.182       \\
            & 160 & 2.47e-02 & 1.015 & 1.85e-03 & 1.858  &     4.28e-04 & 2.493       \\
            & 320 & 1.23e-02 & 1.008 & 4.81e-04 & 1.943  &     6.47e-05 & 2.727       \\
            & 640 & 6.13e-03 & 1.004 & 1.21e-04 & 1.996  &     8.86e-06 & 2.868       \\
            \hline   
		\end{tabular}
	\end{small}	
	\caption{\label{tab:1d_adv} $L_{\infty}$-errors and orders of accuracy for Example \ref{ex:1dAdC} at $T = 0.5$. }
\end{table}





\paragraph{2D examples} Next we present some two-dimensional experiments here and demonstrate that our kernel-based approximations to the differential operator can deal with 2D Hamilton-Jacobian equations and wave equations which are involved with the first derivative operator and second derivative operator separately. 

\begin{exa}\label{ex:adv}
    We test the accuracy of the scheme for the two-dimensional convection equation
    \begin{align}
\left\{\begin{array}{ll}
u_{t} + (u_{x} + u_{y}) = 0, & -2 \leq x,y \leq 2,\\
u(x,y,0)= -\cos (\pi(x+y)/2),
\end{array}
\right.
\end{align}
in the Hamilton-Jacobian form with Dirichlet boundary conditions   $u(-2,y,t)=-\cos(\pi((-2+y)/2-t))$
and $u(x,-2,t)=-\cos(\pi((x-2)/2-t))$. The exact solution is $u(x,y,t)=-\cos(\pi((x+y)/2-t))$.
In Table \ref{tab:adv}, we test the problem in uniform meshes and present the $L_\infty$ errors and the orders of accuracy, which show the desired results.
\end{exa}

\begin{table}[htbp]
	\centering
\vspace{2mm}
	\begin{small}
		\begin{tabular}{|c|c|cc|cc|cc|}
			\hline
			\multirow{2}{*}{CFL} &  \multirow{2}{*}{$N$} & \multicolumn{2}{c|}{$k=1$.} & \multicolumn{2}{c|}{$k=2$.} & \multicolumn{2}{c|}{$k=3$.}\\
			\cline{3-8}
			& &  error &   order  &  error &  order  &  error  & order  \\\hline
			\multirow{5}{*}{0.5} 
            &$ 20 \times  20$ &  1.83e-02 &   --   &  1.76e-01 &   --   & 3.14e-02 &   --    \\
            &$ 40 \times  40$ &  4.50e-03 & 2.025  &  4.56e-02 & 1.946  & 3.00e-03 & 3.390   \\
            &$ 80 \times  80$ &  1.12e-03 & 2.008  &  1.16e-02 & 1.981  & 3.18e-04 & 3.235   \\
            &$160 \times 160$ &  2.79e-04 & 2.003  &  2.91e-03 & 1.991  & 4.48e-05 & 2.827   \\
            &$320 \times 320$ &  6.98e-05 & 2.001  &  7.29e-04 & 1.995  & 5.48e-06 & 3.033   \\ 
            \hline   
			\multirow{5}{*}{1} 
            &$ 20 \times  20$ &  7.29e-02 &   --   &  5.89e-01 &   --   &  1.28e-01 &   --     \\
            &$ 40 \times  40$ &  1.80e-02 & 2.018  &  1.75e-01 & 1.749  &  1.92e-02 & 2.737    \\
            &$ 80 \times  80$ &  4.47e-03 & 2.008  &  4.56e-02 & 1.942  &  2.60e-03 & 2.880    \\
            &$160 \times 160$ &  1.12e-03 & 2.002  &  1.16e-02 & 1.981  &  3.37e-04 & 2.950    \\
            &$320 \times 320$ &  2.79e-04 & 2.000  &  2.91e-03 & 1.991  &  4.28e-05 & 2.978    \\ 
            \hline    
			\multirow{5}{*}{2}  
            &$ 20 \times  20$ &  3.06e-01 &   --   &  1.15e+00 &   --   &   3.04e-01 &   --      \\
            &$ 40 \times  40$ &  7.25e-02 & 2.078  &  5.89e-01 & 0.971  &   5.08e-02 & 2.582     \\
            &$ 80 \times  80$ &  1.79e-02 & 2.016  &  1.75e-01 & 1.747  &   8.96e-03 & 2.504     \\
            &$160 \times 160$ &  4.47e-03 & 2.003  &  4.56e-02 & 1.943  &   1.27e-03 & 2.819     \\
            &$320 \times 320$ &  1.12e-03 & 2.001  &  1.16e-02 & 1.981  &   1.68e-04 & 2.917     \\ 
            \hline   
		\end{tabular}
	\end{small}	
	\caption{\label{tab:adv} $L_{\infty}$-errors and orders of accuracy for Example \ref{ex:adv} at $T=2$. }
\end{table}

\begin{exa} \label{ex:squre_wave}
 Non-uniform mesh two-dimensional linear diffusive equation with Dirichlet boundary condition: 
\begin{align}
\left\{\begin{array}{ll}
u_{t} = q (\,u_{xx} + \,u_{yy}), & 0 < x < 1,\\
u(x,0)= f(x),
\end{array}
\right.
\end{align}
with the Dirichlet boundary condition:$$u(0,y,t) = e^{y+t}, u(1,y,t) = e^{1+y+t},u(x,0,t) = e^{x+t}, u(x,1,t) = e^{1+x+t},$$
Here, $q\geq0$ is a given constant. Let $q = 1/2$, $f(x) = \sin(\pi x)\sin(\pi y) + e^{x+y}$, then the exact solution is: 
$$u^{e}(x,y,t)= \sin(\pi x)\sin(\pi y)e^{-\pi^2t} + e^{x+y+t}.$$
\end{exa} 

We use a non-uniform mesh here to do the simulation. In this mesh, he first and last cells along each line are half the length of the middle cells with all the middle cells are uniform in size; see left graph in Figure \ref{fig:square_numerical}. The success of this example can help us deal with simulations on meshes with more complex geometry.  In Table \ref{tab:squre_wave}, it is observed that the proposed scheme achieves the appropriate convergence orders for $k = 2$ and $k = 3$ as it does for the one-dimensional non-uniform mesh case in Example \ref{ex:1d-wave}. And for $k = 1$, first order accuracy is achieved as the scheme expected in Sec. \ref{sec:newbdyterm} and due to the non-uniform mesh, superconvergence is not achieved. 
To enhance clarity, in Figure \ref{fig:square_numerical}, we illustrate the mesh utilized in our simulations and present the numerical solutions using a grid resolution of $N = 320 \times 320$ with $k = 3$. 

\begin{table}[htbp]
	\centering
\vspace{2mm}
	\begin{small}
		\begin{tabular}{|c|c|cc|cc|cc|}
			\hline
			\multirow{2}{*}{CFL} &  \multirow{2}{*}{$N$} & \multicolumn{2}{c|}{$k=1$.} & \multicolumn{2}{c|}{$k=2$.} & \multicolumn{2}{c|}{$k=3$.}\\
			\cline{3-8}
			& &  error &   order  &  error &  order  &  error  & order  \\\hline
			\multirow{5}{*}{0.5} 
            &$20 \times 20$ &  5.97e-02 &   --  &     1.56e-03 &   --   &    4.65e-04 &   --     \\
            &$40 \times 40$ &   2.99e-02 & 0.997   &   4.12e-04 & 1.920       &  8.45e-05 & 2.460        \\
            &$80 \times 80$ &  1.50e-02 & 0.998   &    1.09e-04 & 1.919       &   1.30e-05 & 2.695       \\
            & $160 \times 160$ &  7.49e-03 & 0.999   &   2.83e-05 & 1.945        &  1.82e-06 & 2.838        \\
            & $320 \times 320$ &  3.74e-03 & 1.000   &   7.24e-06 & 1.966        &  2.48e-07 & 2.879        \\ 
            \hline   
			\multirow{5}{*}{1} 
            & $20 \times 20$ &  1.19e-01 &   --  &   6.41e-03 &   --    &  2.09e-03 &   -- \\
            & $40 \times 40$ &  5.97e-02 & 0.994   &   1.56e-03 & 2.035     &  4.71e-04 & 2.152  \\
            & $80 \times 80$ &  2.99e-02 & 0.997   &   4.12e-04 & 1.924     &   8.51e-05 & 2.469     \\
            & $160 \times 160$  &  1.50e-02 & 0.998   &  1.09e-04 & 1.920     &   1.31e-05 & 2.698    \\
            & $320 \times 320$ &   7.49e-03 & 0.999  &  2.83e-05 & 1.945     &   1.81e-06 & 2.855    \\
            \hline    
			\multirow{5}{*}{2}  
            &$20 \times 20$ &    2.36e-01 &   --   &   3.19e-02 &   --   &      7.02e-03 &   --     \\
            &$40 \times 40$ &    1.19e-01 & 0.988    &   6.43e-03 & 2.311     &   2.10e-03 & 1.740        \\
            &$80 \times 80$ &    5.97e-02 & 0.994    &   1.56e-03 & 2.039    &    4.71e-04 & 2.156        \\
            & $160 \times 160$ &   2.99e-02 & 0.997     &  4.12e-04 & 1.924    &    8.51e-05 & 2.470        \\
            & $320 \times 320$ &   1.50e-02 & 0.998     &  1.09e-04 & 1.921    &    1.32e-05 & 2.693        \\
            \hline   
		\end{tabular}
	\end{small}	
	\caption{\label{tab:squre_wave} $L_{\infty}$-errors and orders of accuracy for Example \ref{ex:squre_wave} at $T=0.5$. }
\end{table}

\begin{figure}[htbp]
    \centering
    \includegraphics[width=0.35\textwidth]{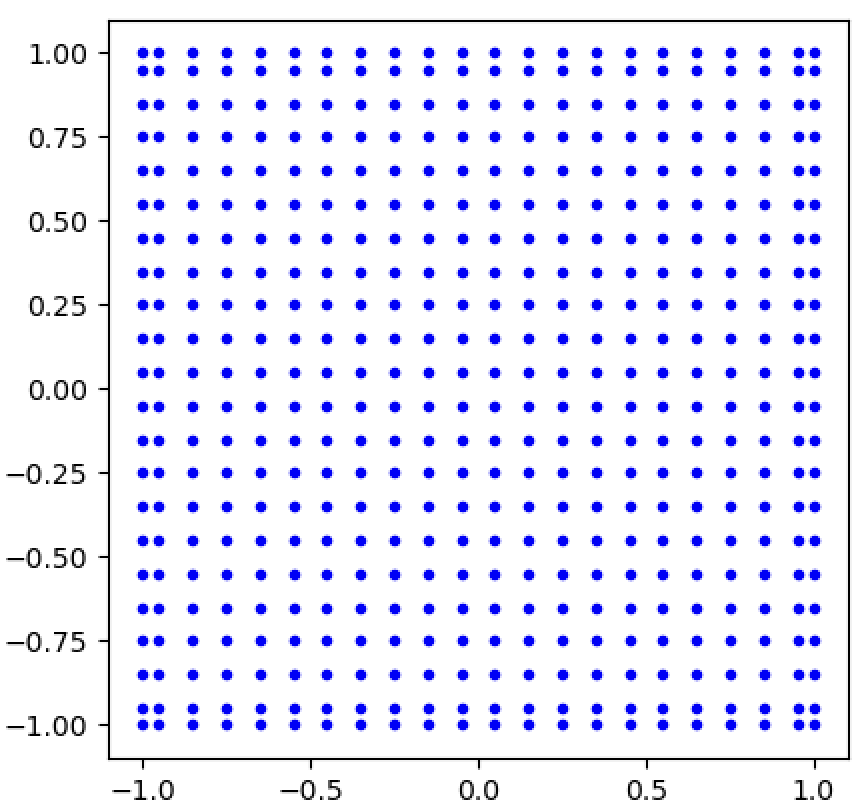}
    \includegraphics[width=0.5\textwidth]{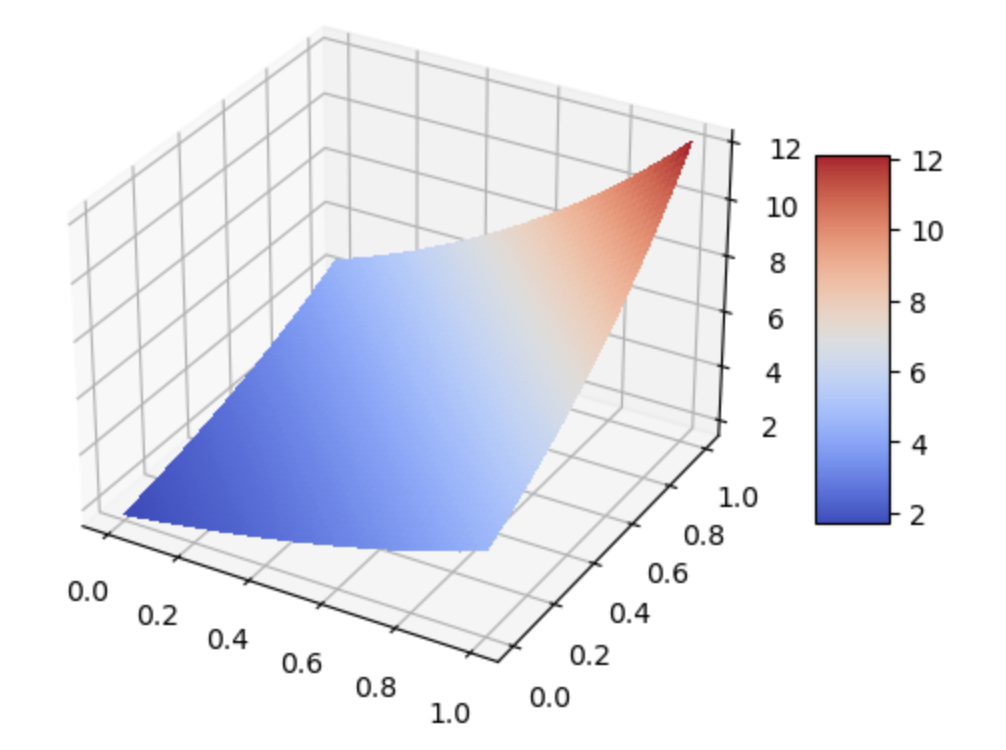}
    \caption{The square nonuniform mesh used and numerical solution of Example \ref{ex:squre_wave} with third order scheme.}
    \label{fig:square_numerical}
\end{figure}

\begin{exa}\label{ex:drumhead}
We test the accuracy of the proposed scheme for the two-dimensional drumhead wave equation with Dirichlet boundary condition. The exact solution is given by separation of variables and for more details, see \cite[Chapter 4.3]{asmar2016partial}.
\begin{align}
\left\{\begin{array}{ll}
u_{tt} = \Delta u, & (x,y) \in \Omega := \{(x,y) \in \mathbb{R}^2: x^2 +y^2 \le 1\},\\
u(x,y,0)= f(x,y,0), \quad u_t(x,y,0) = g(x,y,0), \\
u(\partial \Omega, t) = 0
\end{array}
\right.
\end{align}

(1) For $f(x,y,0) = J_0(k_1 \sqrt{x^2+y^2}), \quad g(x,y,0) = 0,$ where $k_1$ is the first positive zero of the Bessel function of the first kind of order 0, $J_0(x)$.

Our first test is to show that our scheme can solve the problem with the expected accuracy/order. In Table \ref{tab:drumhead}, using a of CFL = 0.5, 1 and 2, we provide the $L_{\infty}$-errors and orders of accuracy for schemes with $k = 1,$ 2 and 3, which are shown to be $k-$th order. 

(2) For $f(x,y,0) = J_0(k_2 \sqrt{x^2+y^2}), \quad g(x,y,0) = 0,$ where $k_2$ is the second positive zero of Bessel function $J_0(x)$.

The numerical solutions using a grid resolution of $N = 200 \times 200$ with $k = 3$ is shown for several times levles in Figure \ref{fig:drumhead_numerical}. As long as the spatial resolution is given appropriately, the scheme can accurately capture the wave shape and we also remark here  that a large CFL can be chosen to solve the problem, due to the  unconditional stability of the method. 

We also present the numerical solutions for both initial conditions at the final time $T = 1$ with a grid of $N = 100 \times 100$ with $k = 3$ in Figure \ref{fig:drumhead2}. 
\end{exa}

\begin{table}[htbp]
	\centering
\vspace{2mm}
	\begin{small}
		\begin{tabular}{|c|c|cc|cc|cc|}
			\hline
			\multirow{2}{*}{CFL} &  \multirow{2}{*}{$N$} & \multicolumn{2}{c|}{$k=1$.} & \multicolumn{2}{c|}{$k=2$.} & \multicolumn{2}{c|}{$k=3$.}\\
			\cline{3-8}
			& &  error &   order  &  error &  order  &  error  & order  \\\hline
			\multirow{5}{*}{0.5} 
            & $20 \times 20$ & 4.53e-02 &   -- &  2.28e-02 &   --&   6.55e-03 &   --   \\
            & $40 \times 40$ & 2.02e-02 & 1.168 &  6.38e-03 & 1.835&   1.14e-03 & 2.519    \\
            & $80 \times 80$ & 9.36e-03 & 1.107 &  1.69e-03 & 1.919&   1.95e-04 & 2.548    \\
            & $160 \times 160$ &4.48e-03 & 1.062 & 4.33e-04 & 1.963 &   3.20e-05 & 2.612    \\
            & $320 \times 320$ & 2.19e-03 & 1.034&  1.09e-04  & 1.984 &      5.02e-06 & 2.669 \\ 
            \hline   
			\multirow{5}{*}{1} 
            & $20 \times 20$ &   1.08e-01 &   --    &   7.43e-02 &   --   &   3.26e-02 & --   \\
            &  $40 \times 40$ &   4.55e-02 & 1.240    &   2.27e-02 & 1.707   &  6.28e-03 & 2.374    \\
            & $80 \times 80$ &   2.02e-02 & 1.173    &    6.35e-03 & 1.841  &   1.11e-03 & 2.501   \\
            &  $160 \times 160$ &   9.36e-03 & 1.109   &  1.67e-03 & 1.922    & 1.91e-04 & 2.538     \\
            &  $320 \times 320$ &  4.48e-03 & 1.063    &  4.29e-04 & 1.965    & 3.14e-05 & 2.604      \\
            \hline    
			\multirow{5}{*}{2}  
            & $20 \times 20$ &    2.61e-01 &   --  & 2.11e-01 &   --&1.13e-01 &   --  \\
            &  $40 \times 40$ &   1.08e-01 & 1.278   &   7.35e-02 & 1.523    &   3.13e-02 & 1.860   \\
            & $80 \times 80$ &    4.55e-02 & 1.241  &    2.25e-02 & 1.709   &    6.07e-03 & 2.366  \\
            &  $160 \times 160$ &   2.02e-02 & 1.173  &  6.26e-03 & 1.844  &     1.08e-03 & 2.492 \\
            &  $320 \times 320$ &   9.36e-03 & 1.109  &  1.65e-03 & 1.926    &   1.87e-04 & 2.529   \\
            \hline   
		\end{tabular}
	\end{small}	
	\caption{\label{tab:drumhead} $L_{\infty}$-errors and orders of accuracy for Example \ref{ex:drumhead} (1) at $T = 1$. }
\end{table}

\begin{figure}[htbp]
    \centering
    \includegraphics[width=0.9\textwidth]{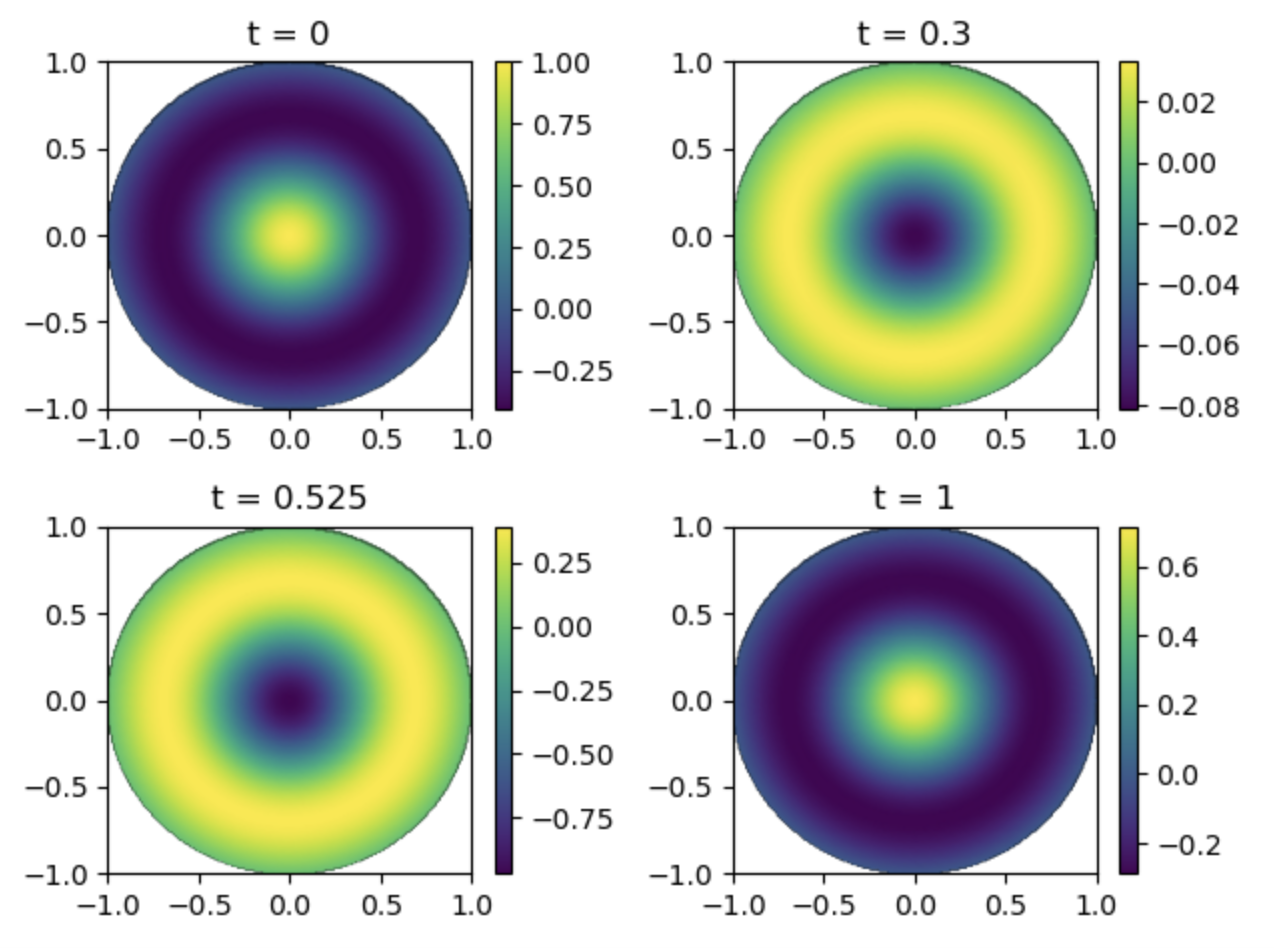}
    \caption{Numerical solution of Example \ref{ex:drumhead} (2) with different time level with third order scheme.}
    \label{fig:drumhead_numerical}
\end{figure}

\begin{figure}[htbp]
    \centering
    \includegraphics[width=0.45\textwidth]{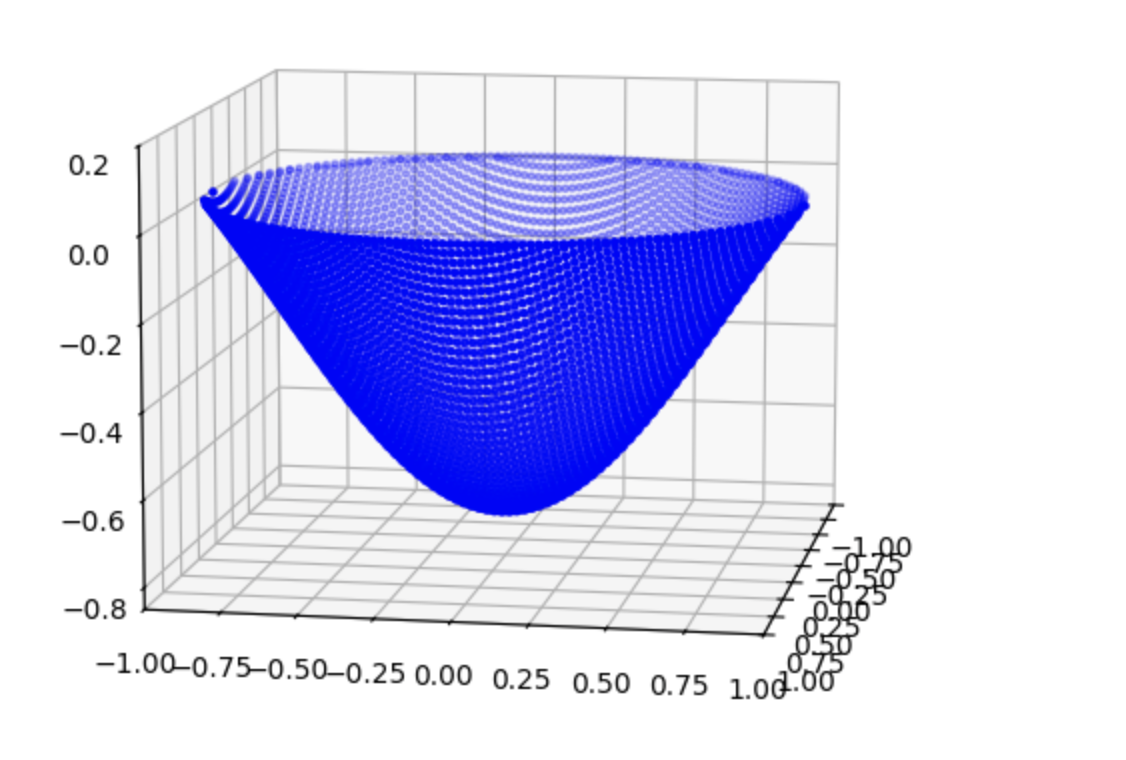}
    \includegraphics[width=0.45\textwidth]{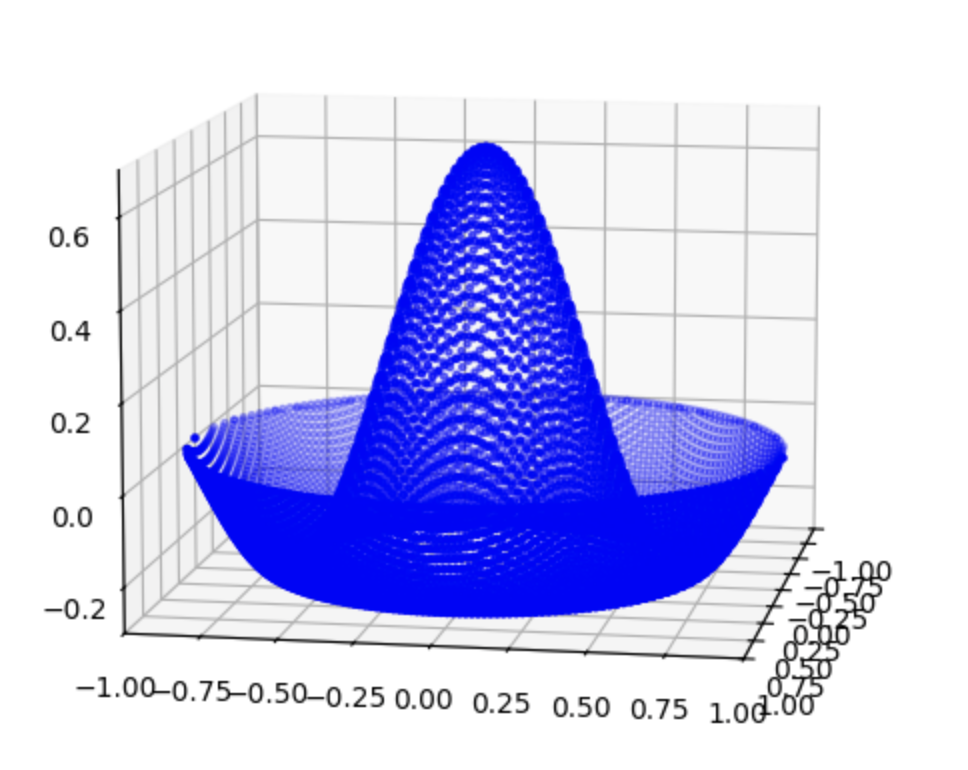}
    \caption{Numerical solution of Example \ref{ex:drumhead} (1) with the first eigen-mode (left) and Example \ref{ex:drumhead} (2) with the second eigen-mode (right) at $T = 1$.}
    \label{fig:drumhead2}
\end{figure}

\rev{
\begin{exa}\label{ex:vortex}
    We close this section by presenting a more challenging nonlinear Hamilton-Jacobian example with periodic boundary condition, which is introduced in \cite{bell1989second}. This is also a standard problem to estimate the ability of numerical schemes to resolve thin filaments. 
    \begin{align}
         & u_t + H(u_x,u_y) = 0, -0.5 \le x \le 1.5, -0.25 \le y \le 1.75,  \\
         & u(x,y,0) = \sqrt{(x-0.5)^2+(y-0.75)^2} - 0.15, \nonumber
    \end{align}
    where 
    \[
    H(u,v) = -sin^2(\pi x) sin(2 \pi y)u + sin^2(\pi y) sin(2 \pi x)v. 
    \]
\end{exa}}
\rev{In the following simulations, we apply the proposed third-order solver and present the results by plotting the contours where $u(x,y,t) = 0$ at the selected final time. Firstly, we present the results in Figure~\ref{fig:vortexT1} using a resolution of $N_x \times N_y = 256 \times 256$ at the final time $T = 1$, with CFL numbers 0.5, 1, and 2. Secondly, we perform long-time simulations up to $T = 3$, as shown in Figure~\ref{fig:vortexCFL_T}, using a finer grid with resolution $N_x \times N_y = 1024 \times 1024$ for different CFL values. It is observed that the filament structure of the model is well captured by the scheme. Although the scheme is not perfectly accurate for extremely large CFL numbers, it remains stable in all cases. Moreover, refining the mesh leads to improved results.}
\begin{figure}[htbp]
  \centering

  \begin{subfigure}[b]{0.32\textwidth}
    \includegraphics[width=\linewidth]{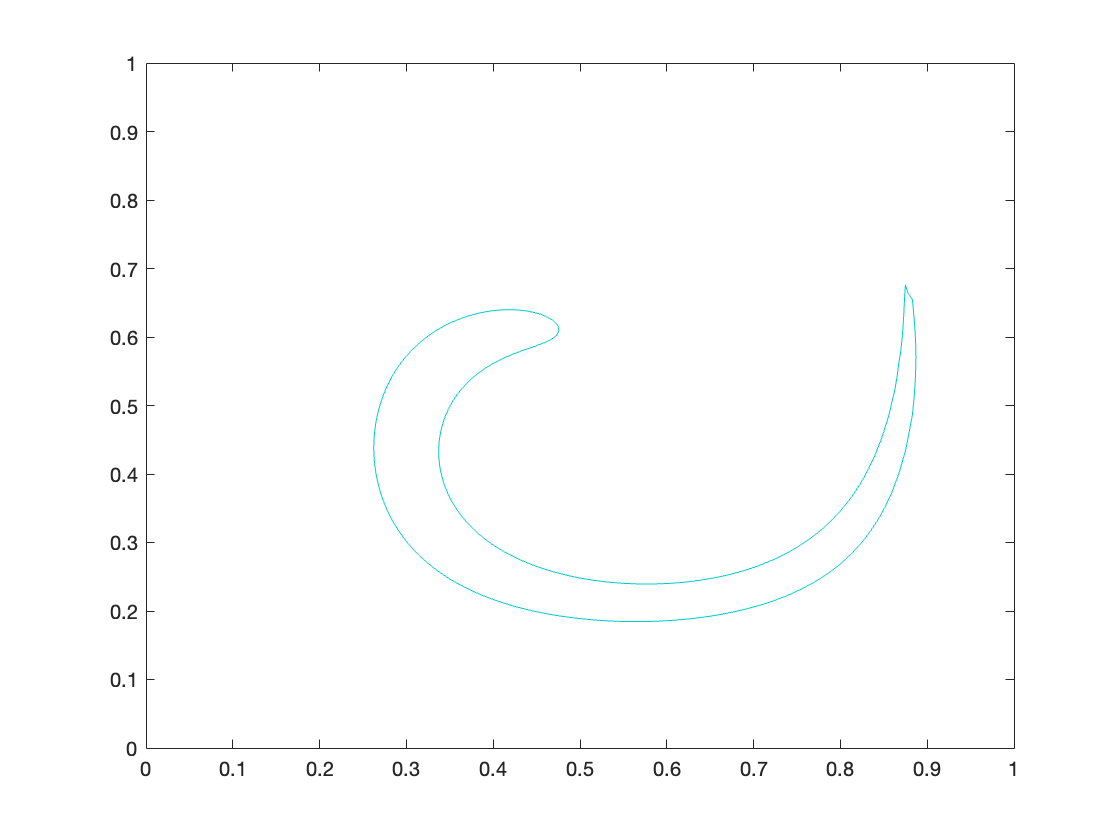}
    \caption{CFL $= 0.5$.} %
  \end{subfigure}
  \begin{subfigure}[b]{0.32\textwidth}
    \includegraphics[width=\linewidth]{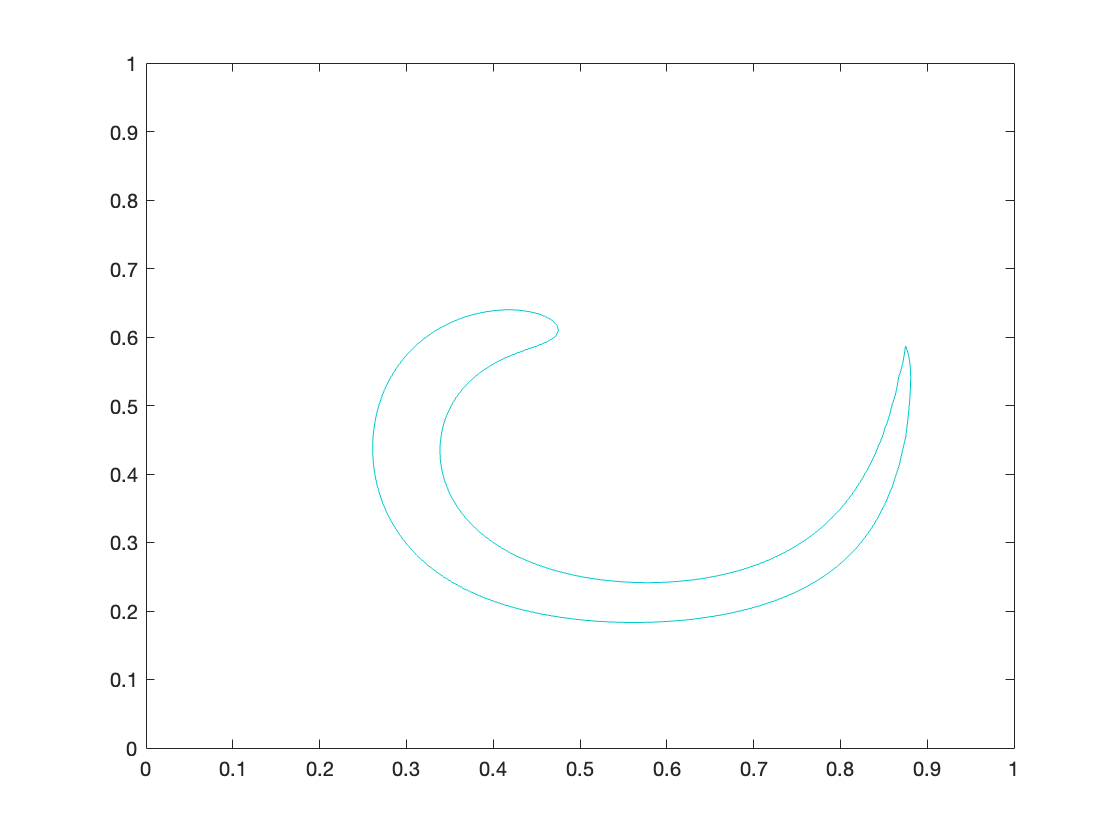}
    \caption{CFL $= 1$.}
  \end{subfigure}
  \begin{subfigure}[b]{0.32\textwidth}
    \includegraphics[width=\linewidth]{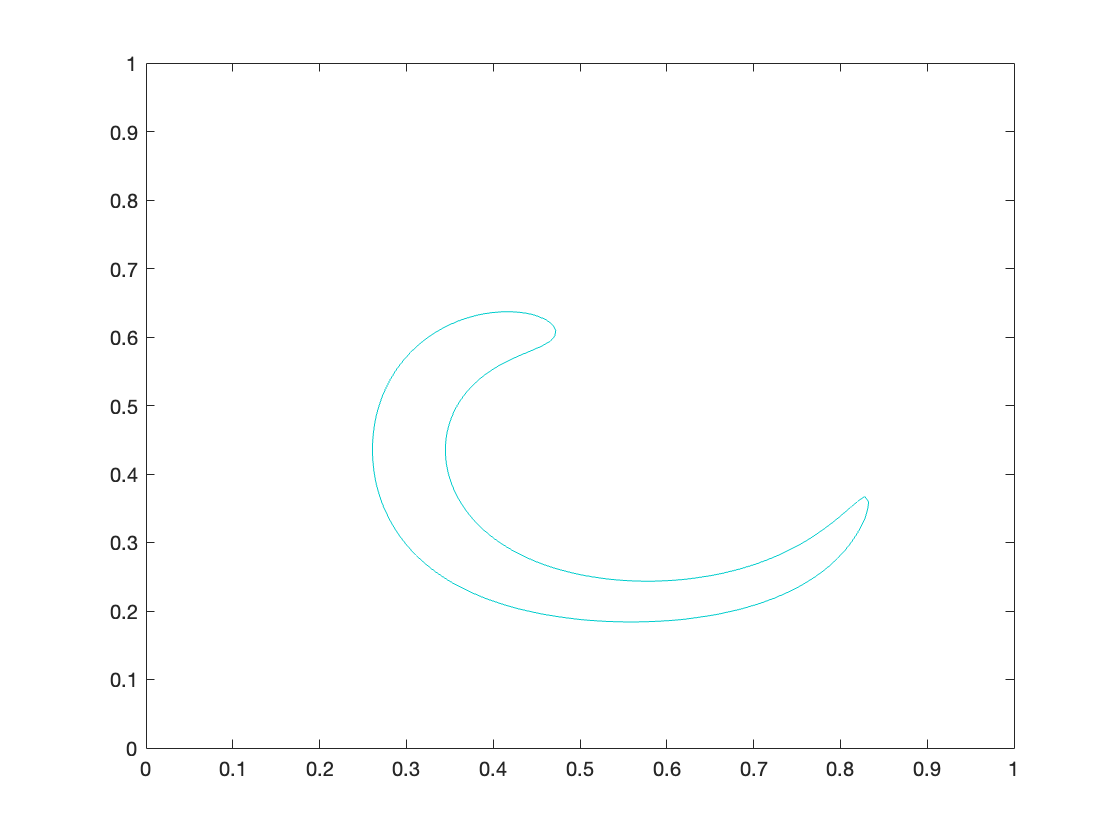}
    \caption{CFL $= 2$.}
  \end{subfigure}
  \caption{The simulations of Example \ref{ex:vortex} with different CFL for the mesh $N_x \times N_y = 256 \times 256$ at time $T = 1$.} 
  \label{fig:vortexT1}
\end{figure}
\begin{figure}[htbp]
  \centering

  \begin{subfigure}[b]{0.32\textwidth}
    \includegraphics[width=\linewidth]{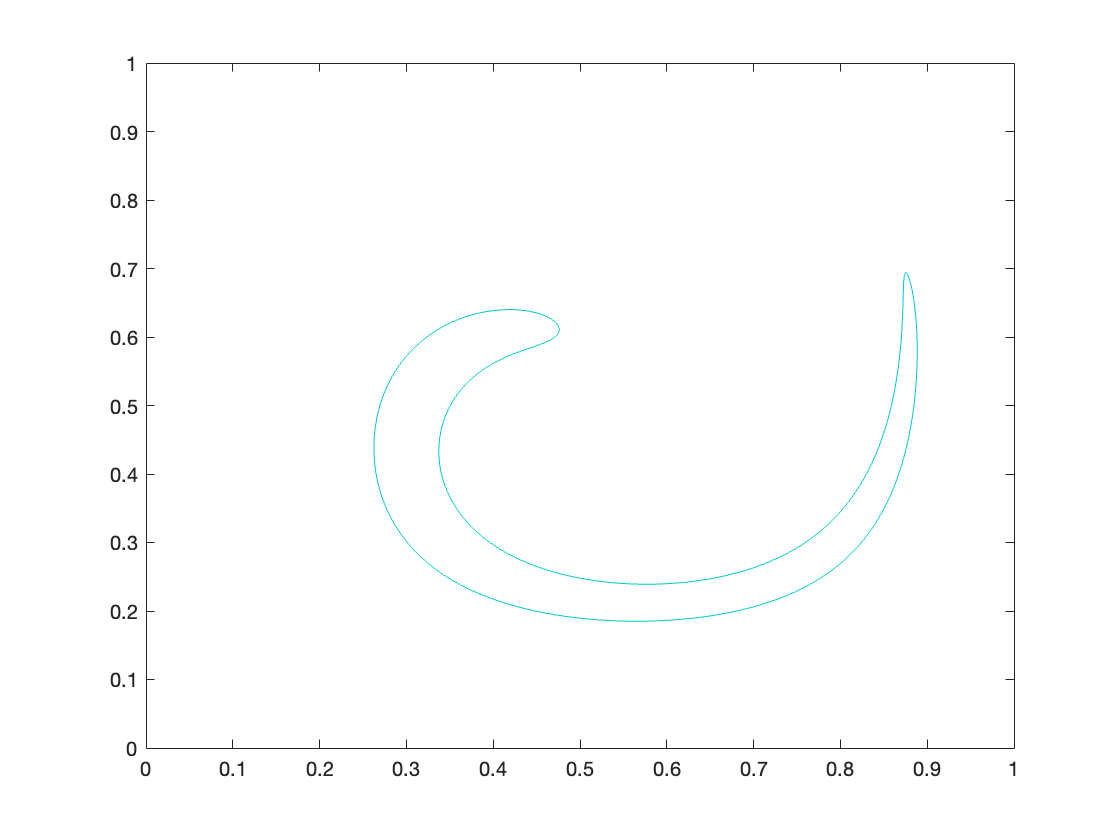}
    \caption{CFL $= 0.5$, $T = 1$}
  \end{subfigure}
  \begin{subfigure}[b]{0.32\textwidth}
    \includegraphics[width=\linewidth]{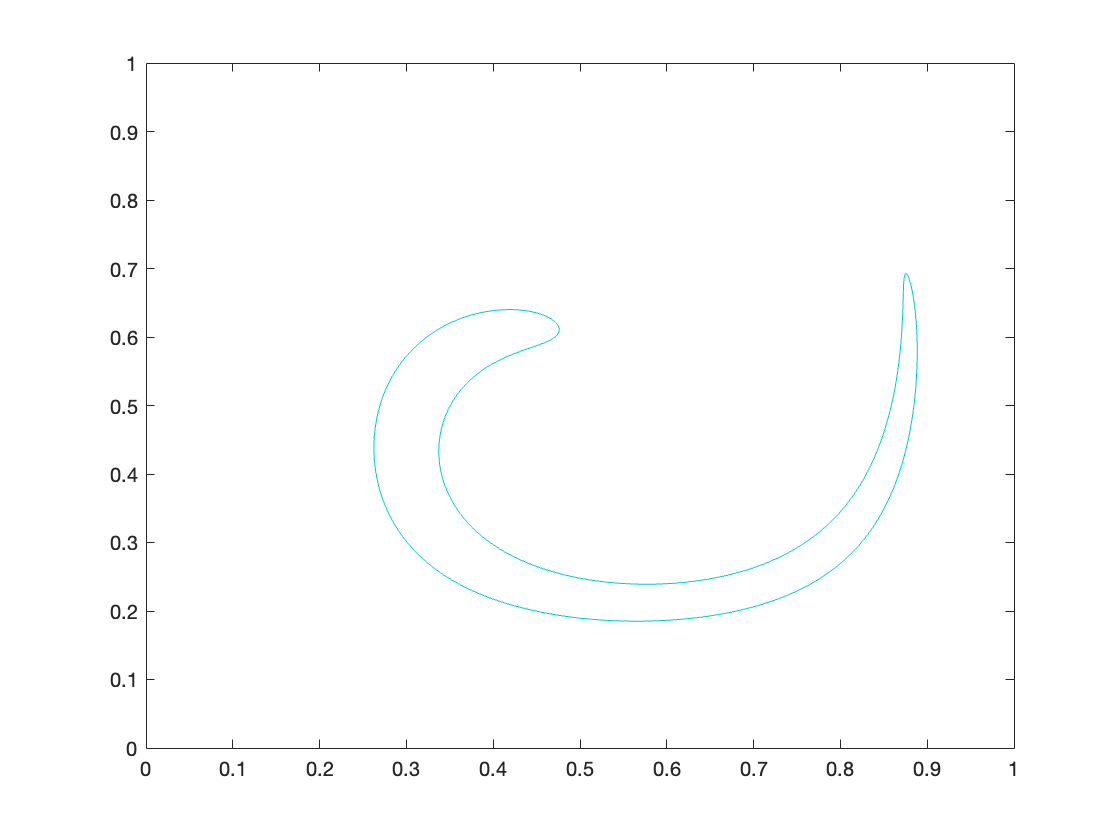}
    \caption{CFL $= 1$, $T = 1$}
  \end{subfigure}
  \begin{subfigure}[b]{0.32\textwidth}
    \includegraphics[width=\linewidth]{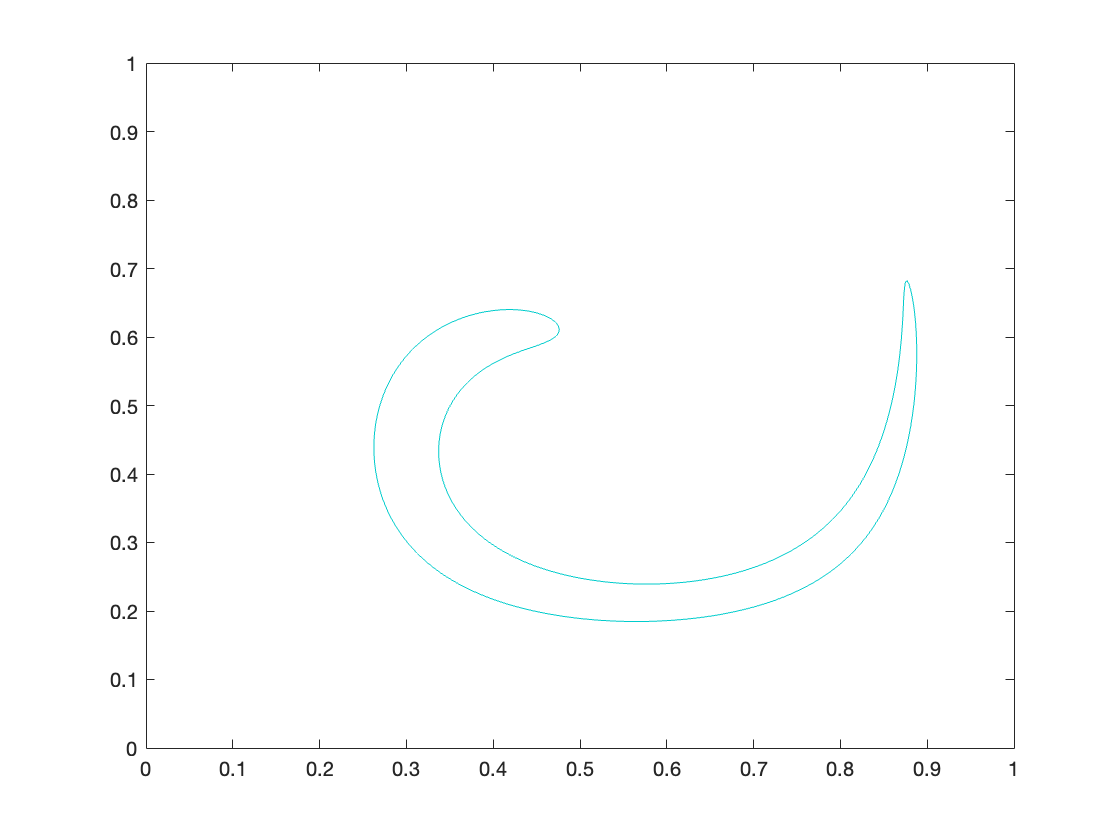}
    \caption{CFL $= 2$, $T = 1$}
  \end{subfigure}

  \par\vspace{1em}

  \begin{subfigure}[b]{0.32\textwidth}
    \includegraphics[width=\linewidth]{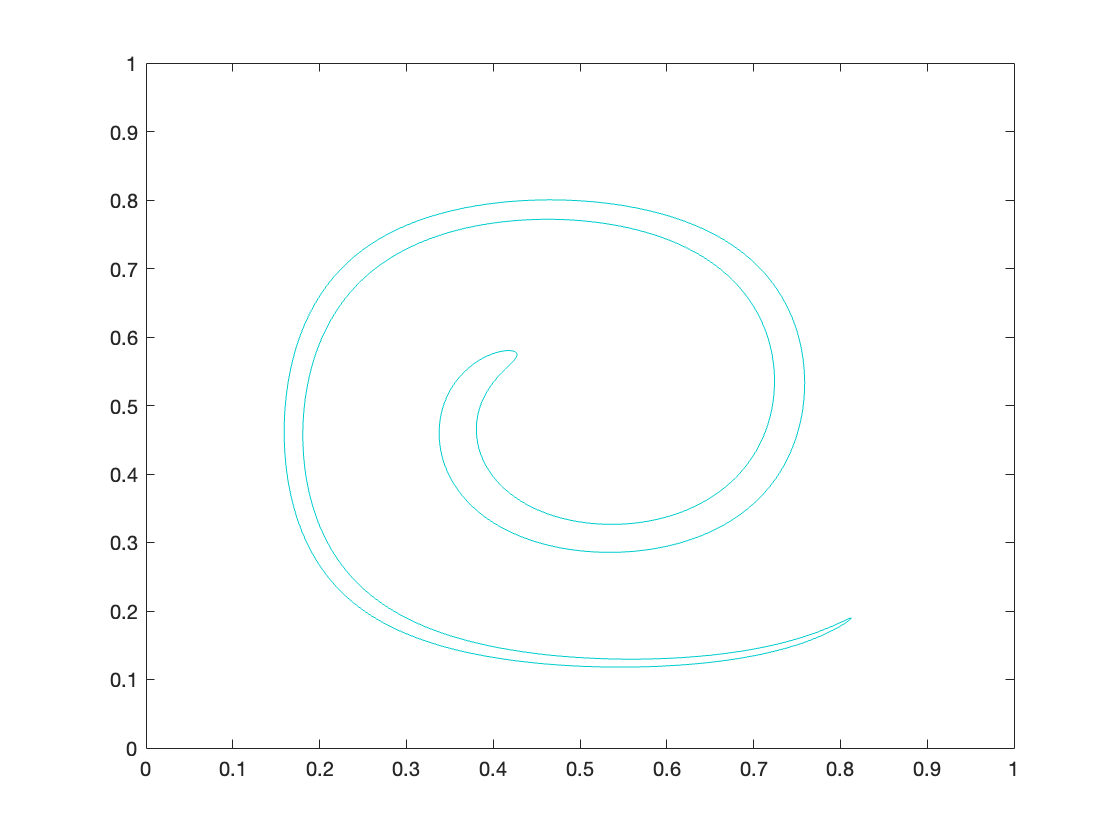}
    \caption{CFL $= 0.5$, $T = 2$}
  \end{subfigure}
  \begin{subfigure}[b]{0.32\textwidth}
    \includegraphics[width=\linewidth]{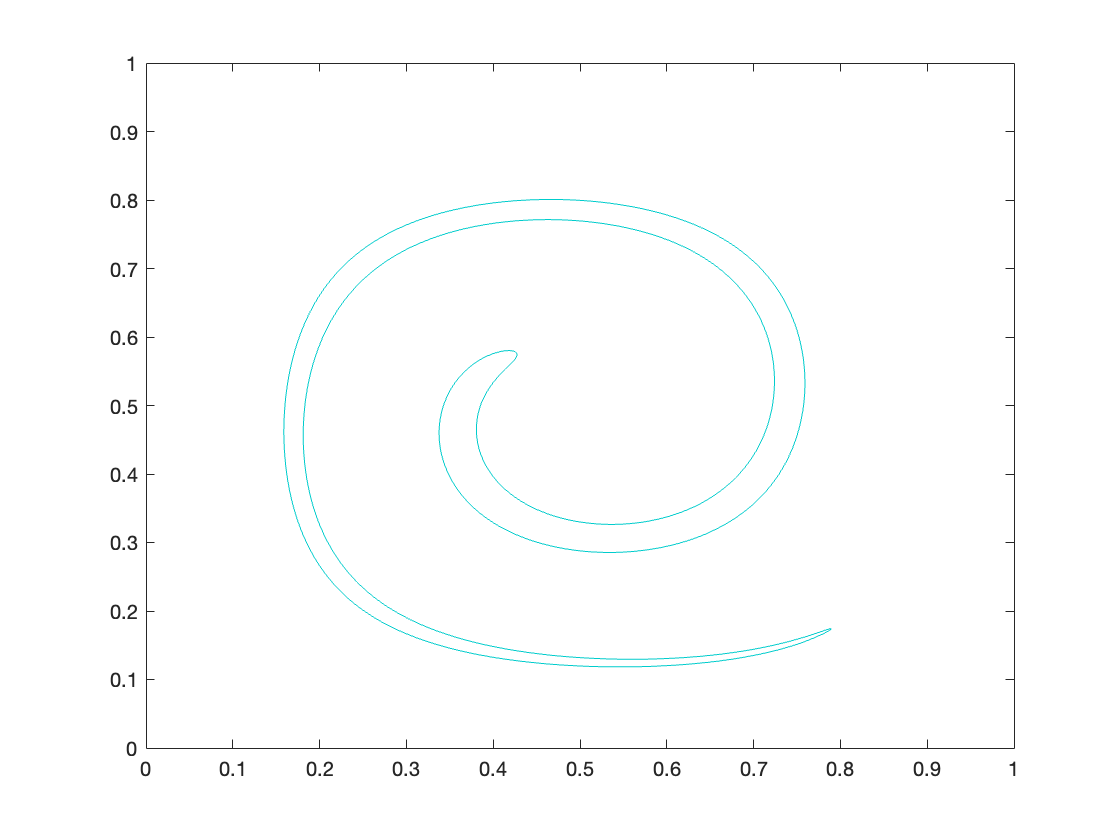}
    \caption{CFL $= 1$, $T = 2$}
  \end{subfigure}
  \begin{subfigure}[b]{0.32\textwidth}
    \includegraphics[width=\linewidth]{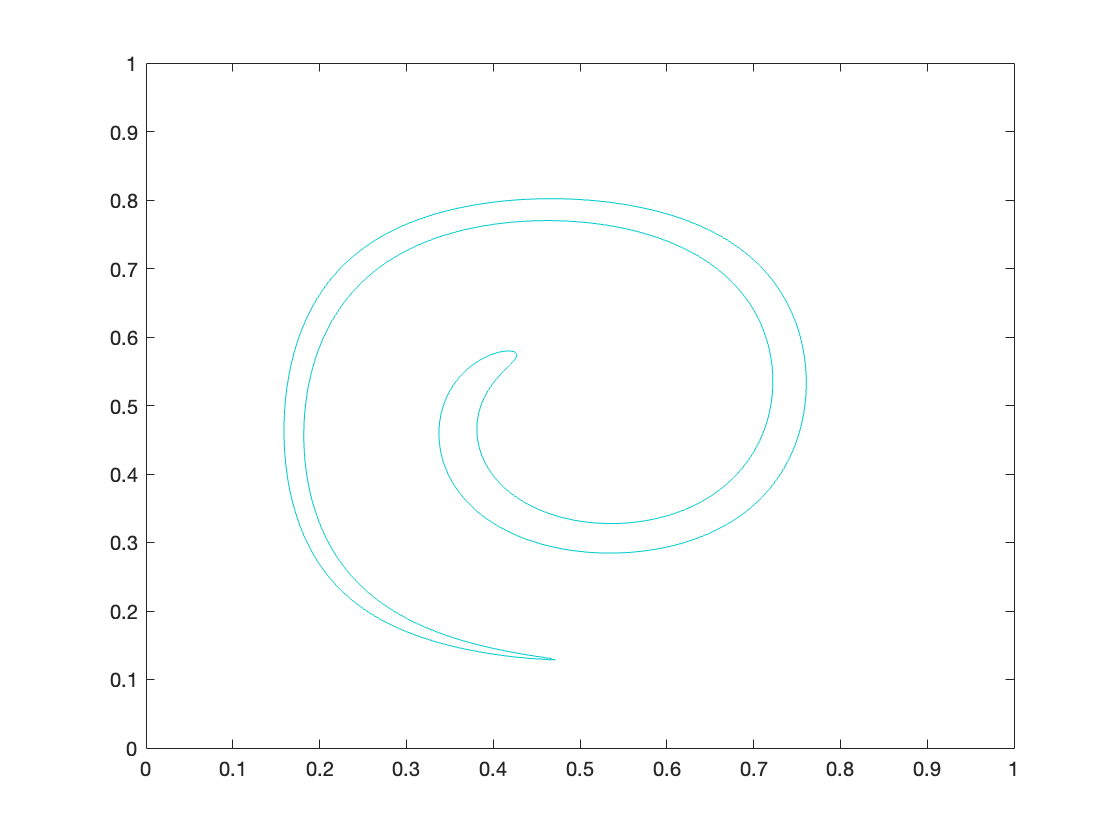}
    \caption{CFL $= 2$, $T = 2$}
  \end{subfigure}

  \par\vspace{1em}

  \begin{subfigure}[b]{0.32\textwidth}
    \includegraphics[width=\linewidth]{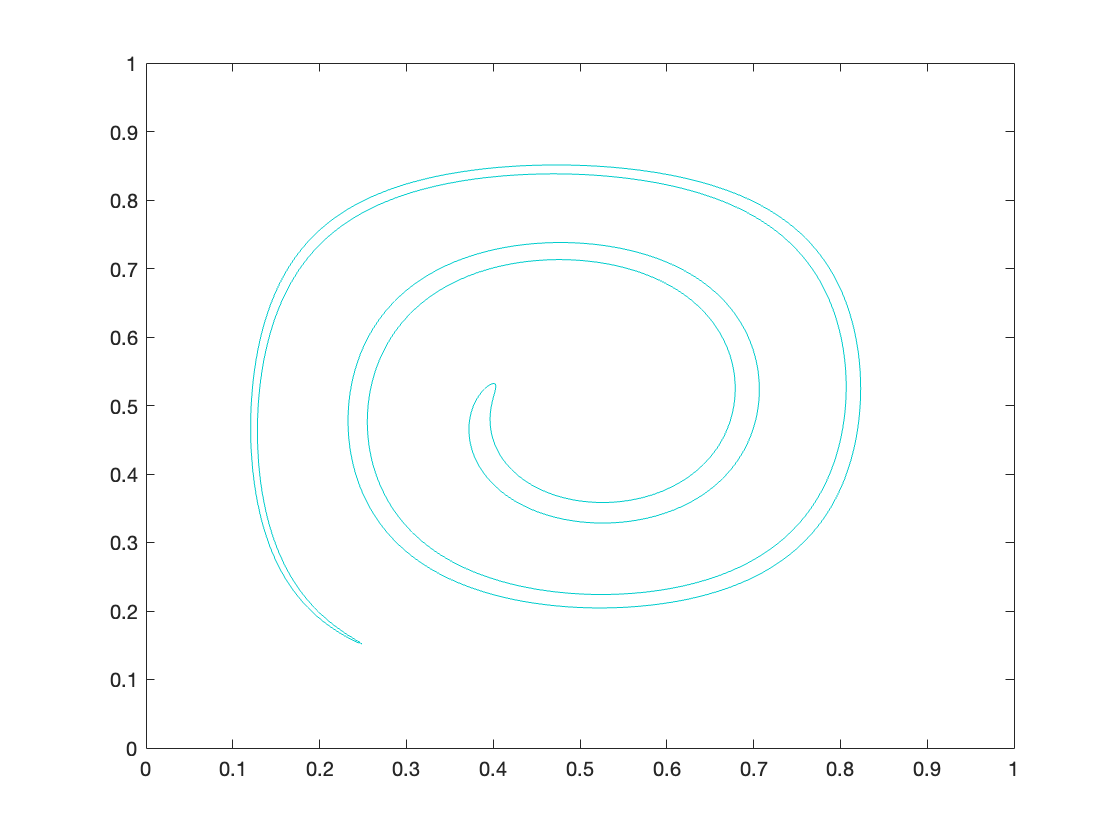}
    \caption{CFL $= 0.5$, $T = 3$}
  \end{subfigure}
  \begin{subfigure}[b]{0.32\textwidth}
    \includegraphics[width=\linewidth]{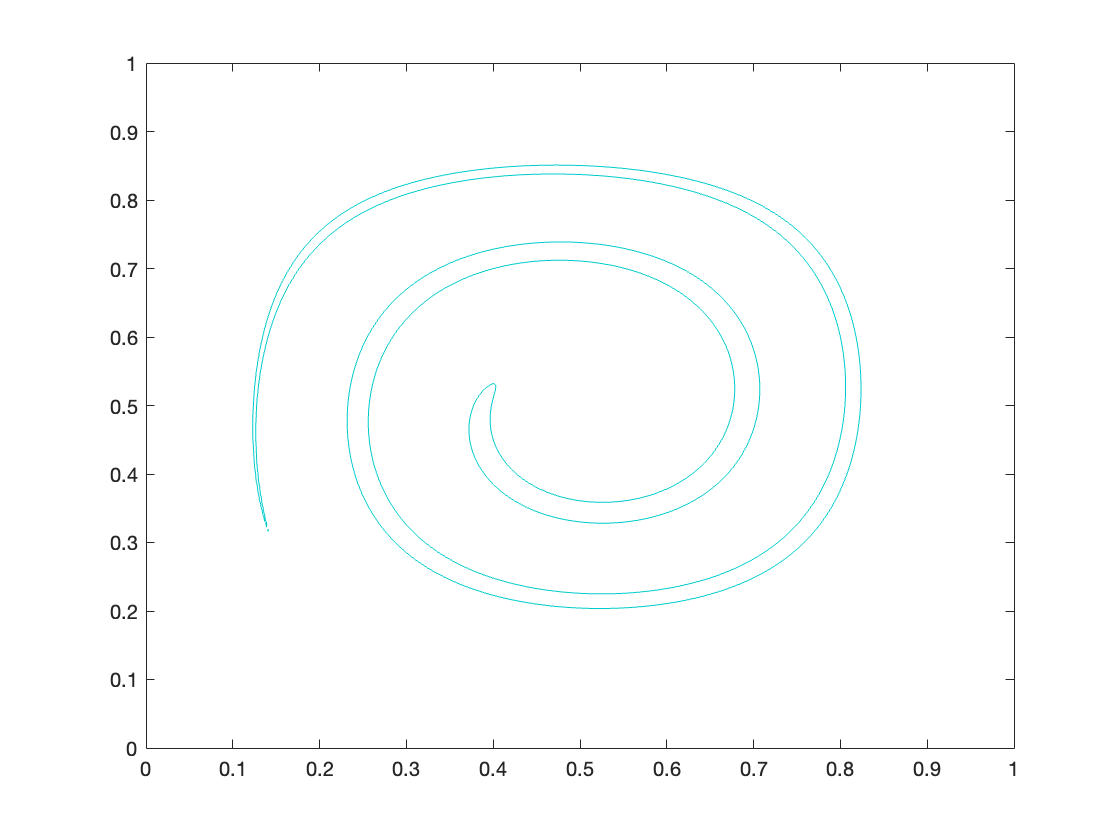}
    \caption{CFL $= 1$, $T = 3$}
  \end{subfigure}
  \begin{subfigure}[b]{0.32\textwidth}
    \includegraphics[width=\linewidth]{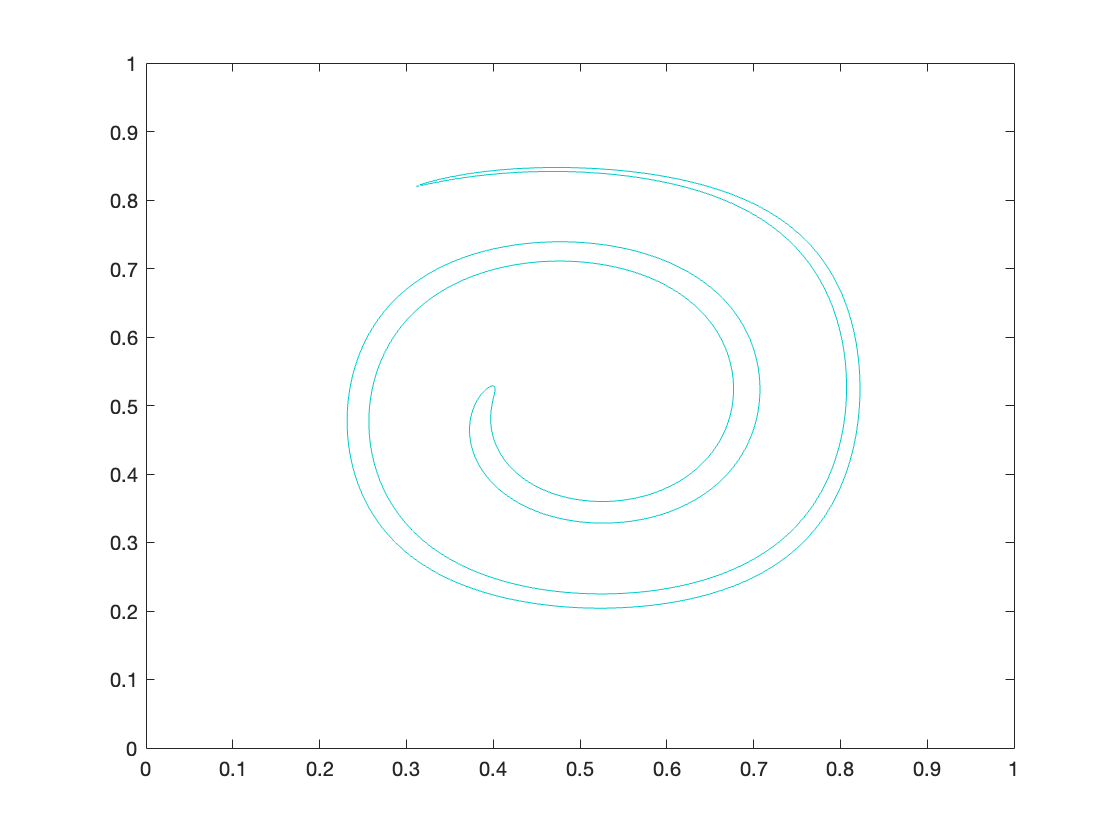}
    \caption{CFL $= 2$, $T = 3$}
  \end{subfigure}

  \caption{Simulations of Example~\ref{ex:vortex} with different CFL numbers on a $1024 \times 1024$ mesh at times $T = 1$, 2, and 3.}
  \label{fig:vortexCFL_T}
\end{figure}

\section{Conclusions}
In this paper, we proposed a novel numerical scheme to solve a variety of vital PDEs involved with first and second spatial derivatives including Hamilton-Jacobian equations, convection-diffusion equations and wave equations. Proposed method, as a generalization of previous high order kernel-based methods on periodic domains, retains their inherent advantages while extending their applicability to various types of boundary conditions. The mathematical proofs and experimental results confirm the theory, demonstrating a systematic approach to achieving high-order accuracy across various boundary conditions. Future work will include a detailed stability analysis and extend the methodology to other systems of equations, such as conservation laws, including the shallow water equations, on more complex geometries.

\bibliographystyle{abbrv}
\bibliography{ref}

\section*{Statements and Declarations}
\subsection*{Funding:} The research of the authors was supported by AFOSR grants FA9550-17-1-0394 and NSF grant DMS-1912183. The first two authors were also supported by AFOSR grants FA9550-24-1-0254, ONR grant N00014-24-1-2242,  and DOE grant DE-SC0023164. The third author was also supported by the National Research Foundation of Korea(NRF) grant funded by the Korea government(MSIT) NRF-2022R1F1A1066389.
 
\subsection*{Competing Interests:} The authors have no relevant financial or non-financial interests to disclose.
\subsection*{Data availability:} All data generated or analyzed during this study are included in this article.

\end{document}